\documentclass[11pt,a4paper,reqno]{amsart}
\fontsize{6.5pt}{6.5pt}\selectfont 
\makeatletter
\def\@evenhead{{\fontsize{6.5pt}{6.5pt}\selectfont \hfil \leftmark\hfil\thepage}}
\def\@oddhead{{\fontsize{6.5pt}{6.5pt}\selectfont\hfil\rightmark \hfil\thepage}}
\makeatother
\usepackage{amsfonts}
\usepackage{booktabs}
\usepackage{amsmath,amssymb,amsthm,amsxtra}
\usepackage{float}
\usepackage{amssymb}
\usepackage{mathabx}
\usepackage{bbm}
\usepackage[colorlinks, linkcolor=blue,anchorcolor=Periwinkle,
citecolor=Red,urlcolor=Emerald]{hyperref}
\usepackage[usenames,dvipsnames]{xcolor}
\usepackage{enumitem}
\usepackage{geometry,array} 
\usepackage{graphicx}
\usepackage{subfigure}
\usepackage{bookmark}
\usepackage{tikz}\usetikzlibrary{matrix}
\usepackage{url}
\usepackage{dsfont}
\usepackage[colorinlistoftodos]{todonotes}
\usepackage{tablists} \restorelistitem

\usetikzlibrary{decorations.markings}
\tikzset{->-/.style={decoration={  markings,  mark=at position #1 with
    {\arrow{>}}},postaction={decorate}}}
\tikzset{-<-/.style={decoration={  markings,  mark=at position #1 with
    {\arrow{<}}},postaction={decorate}}}
\usepackage{extarrows}
\usepackage[all]{xy}
\usepackage{setspace}
\usepackage{thmtools}
\usepackage{thm-restate}
\usepackage{hyperref}
\usepackage{cleveref}
\usepackage{multirow}
\usepackage{mathtools}
\usepackage{autobreak}
\usetikzlibrary{positioning}
\usepackage{xcolor}
\usepackage{graphicx}
\usetikzlibrary{decorations.pathreplacing}
\usetikzlibrary{calc}

\newcommand{\mfw}{\mathbf{w}}
\newcommand{\mfx}{\mathbf{x}}

\newcommand{\mcA}{\mathcal{A}}

\newcommand{\mcC}{\mathcal{C}}

\newcommand{\mcE}{\mathcal{E}}
\newcommand{\mcF}{\mathcal{F}}

\newcommand{\mcK}{\mathcal{K}}

\newcommand{\mcM}{\mathcal{M}}

\newcommand{\mcO}{\mathcal{O}}

\newcommand{\mcT}{\mathcal{T}}

\newcommand{\mcZ}{\mathcal{Z}}
\newcommand{\mbA}{\mathbb{A}}

\newcommand{\mbL}{\mathbb{L}}

\newcommand{\mbN}{\mathbb{N}}

\newcommand{\mbP}{\mathbb{P}}
\newcommand{\mbQ}{\mathbb{Q}}
\newcommand{\mbR}{\mathbb{R}}

\newcommand{\mbT}{\mathbb{T}}

\newcommand{\mbZ}{\mathbb{Z}}

\theoremstyle{plain}
\newtheorem{theorem}{Theorem}[section]

\newtheorem{lemma}[theorem]{Lemma}
\newtheorem{corollary}[theorem]{Corollary}
\newtheorem{proposition}[theorem]{Proposition}
\newtheorem{conjecture}[theorem]{Conjecture}
\theoremstyle{definition}
\newtheorem{definition}[theorem]{Definition}

\newtheorem{example}[theorem]{Example}

\newtheorem{remark}[theorem]{Remark}

\newtheorem{question}[theorem]{Question}

\numberwithin{equation}{section}
\newtheorem{definition-proposition}[theorem]{Definition-Proposition}

\newtheorem{observation}[theorem]{Observation}
\begin{document}

\title {Tropicalization and cluster asymptotic phenomenon of generalized Markov equations}

\date{\today}

\author{Zhichao Chen}
\address{School of Mathematical Sciences\\ University of Science and Technology of China \\ Hefei, Anhui 230026, P. R. China.}
\email{czc98@mail.ustc.edu.cn}
\author{Zelin Jia}
\address{Graduate School of Mathematics\\ nagoya University\\Chikusa-ku\\ Nagoya\\464-8601\\ Japan.}
\email{zelin.jia.c0@math.nagoya-u.ac.jp} 
\maketitle
\begin{abstract} The generalized Markov equations are deeply connected with the generalized cluster algebras of Markov type.
	  We construct a deformed Fock-Goncharov tropicalization for the generalized Markov equations and prove that their tropicalized tree structure is essentially the same as that of the classical Euclid tree. We then define the generalized Euclid tree and prove that it converges to the classical Euclid tree up to a scalar multiple. Moreover, by means of cluster mutations, we exhibit an asymptotic phenomenon, up to some limit $q$, between the logarithmic generalized Markov tree and the classical Euclid tree. A rationality conjecture of $q$ is then put forward.  We also propose a generalized uniqueness conjecture for the generalized Markov equations, which illustrates an application of the asymptotic phenomenon. 
\end{abstract}
\tableofcontents

\section{Introduction}
\subsection{Backgrounds}In 1880, Markov \cite{Mar80} introduced the classical Markov equation 
\begin{align}
	X_1^2+X_2^2+X_3^2=3X_1X_2X_3
\end{align} in connection with problems of Diophantine approximation and the theory of binary quadratic forms. Its positive integer solutions, which are called \emph{Markov triples}, encode the best Diophantine approximations of real numbers by rational numbers in a certain precise sense. The associated \emph{Markov numbers} reveal a deep connection between arithmetic optimization and geometric structures. He defined three transformations as 
\begin{align}
	\begin{array}{cc}
		m_1(X_1,X_2,X_3)=(3X_2X_3-X_1,X_2,X_3),\\
		m_2(X_1,X_2,X_3)=(X_1,3X_1X_3-X_2,X_3),\\
		m_3(X_1,X_2,X_3)=(X_1,X_2,3X_1X_2-X_3).
	\end{array} \label{eq: Markov rules}
\end{align} and proved that all the Markov triples lie in the orbit of the initial solution $(1,1,1)$ under the group $\langle m_1,m_2,m_3 \rangle$. These transformations are also called \emph{mutations} in the sense of cluster algebras.

In 2002, cluster algebras were introduced by Fomin and Zelevinsky \cite{FZ02, FZ03} to investigate the total positivity of Lie groups and canonical bases of quantum groups. In recent years, cluster algebras have developed close connections with various branches of mathematics, such as  representation theory \cite{BMRRT06, DWZ08, BIRS09, HL10, DWZ10, KY11}, higher Teichm{\"u}ller theory \cite{FG09}, integrable system \cite{KNS11}, Poisson geometry \cite{GSV03, GSV12},
	 commutative algebras \cite{Mul13}, combinatorics \cite{FR05, NZ12, LLZ14,  GHKK18} and number theory \cite{Pro20, PZ12, Lam16, Hua22, LLRS23, GM23, BL24, Kau24, CL24, CL25, BL25, Mus25}.
	 
The relations between cluster algebras and Markov equations were firstly discovered by Propp \cite{Pro20}. Later, Peng-Zhang \cite{PZ12}, Huang \cite{Hua22} and Lee-Li-Rabideau-Schiffler \cite{LLRS23} studied some properties and conjectures about the Markov equation, whose Markov triples are in one-to-one correspondence with clusters of the once-punctured torus cluster algebra. Afterwards, Chekhov-Shapiro \cite{CS14} and Nakanishi \cite{Nak15} introduced the generalized cluster algebras. Then, Gyoda and Matsushita \cite{GM23} defined the generalized Markov equations 
\begin{align}
	X_1^2 + X_2^2 + X_3^2 + \lambda_3 X_1 X_2 + \lambda_1 X_2 X_3 + \lambda_2 X_3 X_1 = (3+ \lambda_1 + \lambda_2 + \lambda_3) X_1 X_2 X_3, 
\end{align} where $\lambda_1, \lambda_2, \lambda_3 \in \mbZ_{\geq 0}$. They found the structure of generalized cluster algebras behind and proved that all the \emph{generalized Markov triples} can be generated by the initial solution $(1,1,1)$ through finitely many (generalized) cluster mutations. The tree formed by the generalized Markov triples is referred to be the \emph{generalized Markov tree}.
\subsection{Purposes} The classical Euclid tree $\mcE$ (\Cref{def: classical Euclid}) is a recursive structure with simple additive operation that extends the classical Euclidean algorithm to positive integer triples. It provides a natural framework to study asymptotic relations among integer triples and serves as a combinatorial model of continued fraction-like processes in higher dimensions. According to the generating rules \eqref{eq: Markov rules} and \eqref{eq: classical}, studying their asymptotic correspondence reveals how additive recursive structures may approximate the Markov dynamics, providing a bridge between Euclidean-type algorithms and Diophantine geometry. \emph{The purpose of this paper is to study the tropicalization and asymptotic phenomenon between the generalized Markov tree and the classical Euclid tree via cluster mutations}. For this purpose, we define the $k$-generalized Euclid tree $\mcK$ (\Cref{def: GET}), which is analogue to the classical one. We also extend the well-known \emph{Markov uniqueness conjecture} (\Cref{conj: Markov uniqueness}) to the \emph{generalized Markov uniqueness conjecture} (\Cref{conj: CJ}). We hope that the asymptotic phenomenon may have a good application to the generalized Markov uniqueness conjecture.
\subsection{Main results}	 
 Fock and Goncharov \cite{FG09} introduced the cluster ensembles and the tropicalization, which is called \emph{Fock-Goncharov tropicalization} (see \Cref{prop: trop}). We deform this notion and show the uniformity between the generalized Markov tree and the classical Euclid tree under such tropicalization.
\begin{theorem}[{\Cref{thm: FG trop}}]
	The deformed Fock-Goncharov tropicalization of the generalized Markov tree is the classical Euclid tree.
\end{theorem}
To record the information of mutations, we define $\mfw=[w_1,
  \dots,w_n,\dots]$ to be a reduced mutation sequence, see \Cref{def: reduced mutation}. For the $k$-generalized Euclid tree $\mcK$ and the classical Euclid tree $\mcE$, we introduce the \emph{comparison triple} (\Cref{def: Comparison triple}). Then, the next result is the asymptotic phenomenon between $\mcK$ and $\mcE$.
\begin{theorem}[{\Cref{thm:comparison-convergence}}]
   Take any $k$-generalized Euclid tree $\mcK$ with the initial triple $(A,B,C)$ and the classical Euclid tree $\mcE$ with the initial triple $(a,b,c)$. Let $\mfw=[w_1,
  \dots,w_n,\dots]$ be an infinite reduced mutation sequence. Then, the statements as follows hold:
  \begin{enumerate}
  	\item If $1,2,3$ all appear infinitely many times in $\mfw$, then there exists a real number $q$, such that each component of the triple in $\mcK$ converges to $q$ times of the corresponding component of the triple in $\mcE$ when $n$ goes infinity.
  	\item If one index $i$ of $\{1,2,3\}$ appears only finitely many times in $\mfw$, then there exists a real number $q$, such that two components of the triple indexed by $\{1,2,3\}\backslash \{i\}$ in $\mcK$ converge to $q$ times of the corresponding components of the triple in $\mcE$ when $n$ goes infinity.
  \end{enumerate}
  \end{theorem}
For each generalized Markov triple along $\mfw$, except the initial solution $(1,1,1)$, there exists a unique maximal component. Hence, we associate each triple with a real number, which is the ratio between the maximal one and the product of the other two numbers, see \Cref{sub: Ratio number sequence}. Then, we can get a ratio number sequence $\{k_j\}_{j=1}^{+\infty}$ associated with $\mfw$. It is a strictly increasing sequence (\Cref{lem: ratio increase}) and converges to some real number as follows.

\begin{theorem}[{\Cref{thm: converge}}]
	Let $\mfw=[w_1,\dots,w_n,\dots]$ be an infinite reduced mutation sequence and $\{k_j\}_{j=1}^{+\infty}$ be the ratio number sequence associated with $\mfw$. Then, the following statements hold:
	\begin{enumerate}
		\item If $1,2,3$ all appear infinitely many times in $\mfw$, then $\displaystyle \lim_{j\to +\infty} k_j = 3+\lambda_1+\lambda_2+\lambda_3$.
		\item If one index $i$ of $\{1,2,3\}$ appears only finitely many times in $\mfw$, then there exists a real number $k_{\beta}$, such that $\displaystyle \lim_{j\to +\infty} k_j = k_{\beta}$.
	\end{enumerate}
\end{theorem}
By taking the logarithm of generalized Markov triples, based on the above results, the mutations behave like those of the $k$-generalized Euclid triples. With the help of $k$-generalized Euclid tree $\mcK$ and the ratio number sequence $\{k_j\}$, we can show the asymptotic phenomenon between the logarithmic generalized Markov tree and the classical Euclid tree as follows. 
\begin{theorem}[{\Cref{thm: generalized Markov tree}}]
  Let $\mfw=[w_1,\dots,w_n,\dots]$ be an infinite reduced mutation sequence and $\{k_j\}_{j=1}^{+\infty}$ be the ratio number sequence associated with $\mfw$.
  \begin{enumerate}
  	\item If $1,2,3$ all appear infinitely many times in $\mfw$, then there exists a real number $q$, such that the logarithmic generalized Markov chain along $\mfw$ converges to $q$ times of the corresponding classical Euclid chain when $n$ goes infinity.
  	\item If one index $i$ of $\{1,2,3\}$ appear only finitely many times in $\mfw$, then there exists a real number $q$, such that the components of the logarithmic generalized Markov chain along $\mfw$ indexed by $\{1,2,3\}\backslash \{i\}$ converge to $q$ times of the corresponding components in the classical Euclid chain when $n$ goes infinity.
  \end{enumerate}
\end{theorem}
  In general, it is quite difficult to determine the explicit value of the limit $q$. In fact, it is not even clear whether it is a rational number. However, motivated by \Cref{ex: Fibonacci case}, we propose the following rationality conjecture.
\begin{conjecture}[\Cref{conj: non rational}]
	We conjecture that all such limits $q\in \mbR_+\backslash\mbQ_+$. 
\end{conjecture}
The classical Markov uniqueness conjecture (\Cref{conj: Markov uniqueness}), first proposed by Frobenius \cite{Fro13} in 1913, has remained open for more than a century. It lies at the heart of the study of Diophantine geometry and discrete dynamical systems. As the last main result, we propose a generalized Markov uniqueness conjecture for the generalized Markov equations.
\begin{conjecture}[{\Cref{conj: CJ}}] If $(a,b,c)$ and $(a,b^{\prime},c^{\prime})$ are two positive integer solutions to the generalized Markov equation with $a\geq b\geq c$ and $a\geq b^{\prime}\geq c^{\prime}$, then $b=b^{\prime}$ and $c=c^{\prime}$.
\end{conjecture} As an application of the asymptotic phenomenon, we provide an approximate method for searching for the counter-examples if they exist, see \Cref{subsec: Application}. 
\subsection{Organization of the paper}
This paper is organized as follows. 

In \Cref{Pre}, we review basic definitions and properties about generalized cluster algebras (\Cref{Def of GS} and \Cref{def of GCA}), the generalized Euclid tree (\Cref{def: GET}) and the Fibonacci sequence (\Cref{lem: Fibonacci property} and \Cref{prop: Fibonacci converge}). 

In \Cref{FG trop}, we recall the generalized Markov equation and its relation with generalized cluster algebras (\Cref{thm: generate}). Then, we define the deformed Fock-Goncharov tropicalization of the generalized Markov tree and prove that it is essentially the same as the classical Euclid tree (\Cref{thm: FG trop}).  

In \Cref{comparison}, we define the comparison triple (\Cref{def: Comparison triple}) to compare the $k$-generalized Euclid tree $\mcK$ and the classical Euclid tree $\mcE$ (\Cref{lem: internal condition}).

In \Cref{asym of GET}, with the help of Fibonacci sequence, we prove the boundedness of the comparison triples (\Cref{prop:comparison-bounded}) and the asymptotic phenomenon between the $k$-generalized Euclid tree $\mcK$ and the classical Euclid tree $\mcE$ (\Cref{thm:comparison-convergence}).

In \Cref{asym of LGMT}, we prove the convergency of the ratio number sequence, see \Cref{thm: converge}. Furthermore, we show the asymptotic phenomenon between the logarithmic generalized Markov tree and the classical Euclid tree (\Cref{thm: generalized Markov tree}). We also propose a rationality conjecture about the limit (\Cref{conj: non rational}).

In \Cref{example}, we provide more examples about Lampe's Diophantine equation to exhibit and verify the asymptotic phenomenon, see \Cref{ex: lampe1} and \Cref{ex: lampe2}.

In \Cref{sec: GMUC}, we extend the Markov uniqueness conjecture (\Cref{conj: Markov uniqueness}) to the generalized Markov uniqueness conjecture (\Cref{conj: CJ}). As an application of the asymptotic phenomenon, we give an approximate method to roughly find  the counter-examples if they exist.
\subsection{Conventions} Throughout the paper, we use the following notations.
\begin{itemize}[leftmargin=2em]\itemsep=0pt
\item The integer ring, the set of non-negative integers, the rational number field, the set of positive national numbers, the real number field and the set of positive real numbers are denoted by $\mbZ$, $\mbN$, $\mbQ$, $\mbQ_+$, $\mbR$ and $\mbR_+$ respectively. 
\item We denote by $\text{Mat}_{n\times n}(\mbZ)$ the set of all $n\times n$ integer square matrices. An integer square matrix $B$ is said to be \emph{skew-symmetrizable} if there exists a positive integer diagonal matrix $D$ such that $DB$ is skew-symmetric and $D$ is called the \emph{left skew-symmetrizer} of $B$. 
\item For any $a\in \mbZ$, we denote $[a]_{+}=\max(a,0)$ and then $a=[a]_+-[-a]_+$.
\item For any $x\in \mbR_+$, we use $\log(x)$ to denote the natural logarithm $\log_{e}(x)$.
\item Let $\approx$ be the approximation symbol.  It indicates an informal approximation, meaning that the two quantities are only roughly equal and no rigorous asymptotic relation is intended. 
\end{itemize}

\section{Preliminaries}\label{Pre}
In this section, we recall some basic notions about the generalized cluster algebra, generalized Euclid tree and the Fibonacci sequence.
\subsection{Generalized cluster algebra}
In this subsection, we first recall the definitions and properties about the \emph{generalized cluster algebras} (GCA, for short) based on \cite{Nak15,CS14}.
\begin{definition}[\emph{Generalized seed}]\label{Def of GS}
Let $n\in \mbN_+$ and $\mcF$ be a rational function field of $n$ variables. A \emph{generalized (labeled) seed} is a triple $(\mfx,B,\mcZ)$, where
\begin{itemize}[leftmargin=2em]
\item $\mathbf{x}=(x_1, \dots, x_n)$ is an $n$-tuple of algebraically independent and generating elements of $\mcF$.
\item $B=(b_{ij})_{n\times n}\in \text{Mat}_{n\times n}(\mbZ)$ is a skew-symmetrizable matrix,  
\item $\mcZ=(Z_1,\dots,Z_n)$ is an $n$-tuple of polynomials over $\mbZ_{\geq 0}$, where 
\[Z_i(u)=z_{i,0}+z_{i,1}u+\cdots+z_{i,r_i}u^{r_i},\]
such that $z_{i,0}=z_{i,r_i}=1$. 
\end{itemize} Here, we respectively call $\mfx$  \emph{cluster}, $x_i$ \emph{cluster variable}, $B$ \emph{exchange matrix}, $Z_i$ \emph{exchange polynomial} and $r_i$ \emph{exchange degree}. Moreover, let $R=diag(r_1,\dots, r_n)$. Then, it is a positive integer diagonal matrix, which is called an \emph{exchange degree matrix}.  
\end{definition}

Note that $BR$ is still a skew-symmetrizable matrix with the skew-symmetrizer $RD$.
\begin{definition}[\emph{Generalized mutation}]
Let $(\mfx,B,\mcZ)$ be a generalized seed and $k\in \{1,\dots,n\}$. We define another generalized seed in direction $k$ by $\mu_{k}(\mfx,B,\mcZ)=(\mfx^{\prime},B^{\prime},\mcZ^{\prime})$, such that 
	\begin{itemize}[leftmargin=2em]
	\item The cluster variables $(x_1^{\prime},\dots,x_n^{\prime})$ are given by 
	\begin{align}\ 
		x_{i}^{\prime}=\left\{
		\begin{array}{ll}
			x_{k}^{-1}\left(\mathop{\prod}\limits_{i=1}^{n} x_i^{[-b_{ik}]_+}\right)^{r_k}Z_k\left(\mathop{\prod}\limits_{i=1}^{n} x_i^{b_{ik}}\right), &   \text{if}\ i=k, \\
			x_{i}, &  \text{if}\ i \neq k, 
		\end{array} \right.
	\end{align}

	\item The entries of $B^\prime=(b^\prime_{ij})_{n\times n}$ are given by \begin{align} \label{generalized matrix mutation}
		b_{ij}^{\prime}=\left\{
		\begin{array}{ll}
			-b_{ij}, &  \text{if}\ i=k \;\;\mbox{or}\;\; j=k, \\
			b_{ij}+r_k([b_{ik}]_{+}b_{kj}+b_{ik}[-b_{kj}]_{+}), &  \text{if}\ i\neq k \;\;
			\mbox{and}\; j\neq k. 
		\end{array} \right. 
	\end{align}

	\item The exchange polynomials $\mcZ^{\prime}=(Z^{\prime}_1, \dots, Z^{\prime}_n)$ are given by
		\begin{align}
		Z_{i}^{\prime}(u)=\left\{
		\begin{array}{ll}
			u^{r_k}Z_k(u^{-1}), &  \text{if}\ i=k , \\
			Z_i(u), &  \text{if}\ i\neq k . 
		\end{array} \right. 
	\end{align}

	\end{itemize}

\end{definition}

\begin{remark}\label{no coe GCA}	
Here, for our purpose, we only consider the generalized seeds and generalized mutations without coefficients. We can refer to \cite{CS14, Nak15} for the version with coefficients.
\end{remark}
It can be checked directly that $(\mfx^{\prime},B^{\prime},\mcZ^{\prime})$ is still a generalized seed and $\mu_{k}$ is involutive, that is $\mu_{k}(\mfx^{\prime},B^{\prime},\mcZ^{\prime})=(\mfx,B,\mcZ)$, see \cite{Nak23}. Hence, similar to the classical cluster pattern in \cite{FZ07}, we can define the \emph{generalized cluster pattern} $\mathbf{\Sigma}=\{(\mfx_t,B_t,\mcZ_t)|\ t\in \mbT_n\}$ to be a collection of generalized seeds which are labeled by the vertices of $n$-regular tree $\mbT_{n}$ and connected by a single generalized mutation. Then, we call $n$ the \emph{rank} of the generalized cluster pattern $\mathbf{\Sigma}$.
\begin{definition}[\emph{Generalized cluster algebra}]\label{def of GCA} For a generalized cluster pattern $\mathbf{\Sigma}$, the  \emph{generalized cluster algebra} $\mcA(\mathbf{\Sigma})$ is the $\mbQ$-subalgebra of $\mcF$ generated by all the generalized cluster variables $\{x_{i;t}|\ i=1,\dots,n;t\in\mbT_{n}\}$. 
	
\end{definition}
\begin{remark}
	In particular, when $R=I_n$, the generalized cluster pattern comes back to the classical cluster pattern defined in \cite{FZ02,FZ03,FZ07}. If we denote the mutation of classical cluster algebra by $\mu^{*}$, then there is a well-known result as follows.
\end{remark}
\begin{lemma}[{\cite{Nak15}}]\label{CA GCA mutation} Let $(\mfx,B,\mcZ)$ be a generalized seed and $k\in \{1,\dots,n\}$. Then, the following equality holds 
	\begin{align}
		\mu_{k}(B)R=\mu_k^{*}(BR). \label{eq: compatibility}
		\end{align}
\end{lemma}
\begin{definition}[\emph{Reduced mutation sequence}]\label{def: reduced mutation}
	Consider a sequence ${\bf w}=[w_1,\dots,w_n,\dots]$, where $w_i\in \{1,2,3\}$ for any $i\in \mbN_+$. A sequence $\mfw$ is said to be \emph{reduced} if $w_i \neq w_{i+1}$ for any $i\in \mbN_+$. We denote the set of all reduced sequences by $\mathcal{T}$. If $\mfw=[w_1,w_2,\dots,w_n]$ is finite, then the {\em length} of $\bf w$ is finite, denoted by $|\mfw|<\infty$ and we define its length by $|\mfw|=n$. In particular, we assume that $\mfw=[\ ]=\emptyset$ is also reduced and its length is $0$. If $\mfw$ is infinite, then its length is denoted by $|\mfw|=\infty$. 
\end{definition}
\begin{remark}
	For any infinite reduced sequence ${\bf w}=[w_1,\dots,w_n,\dots]$, we can naturally identify it with a number sequence $\{w_i\}\ (i\geq 1)$. We sometimes denote its finite subsequence $[w_1,\dots,w_n]$ by $\mfw_n$.
\end{remark}
For brevity, let $\mfw=[w_1,w_2,\dots,w_n]\in \mcT$, we always denote the composition of mutations $\mu_{w_n}\circ \dots \circ \mu_{w_1}$ by $\mu^{\mfw}$. That is to say, 
\begin{align}
	\mu^{\mfw}(\mfx,B,\mcZ)=(\mu_{w_n}\circ \dots \circ \mu_{w_1})(\mfx,B,\mcZ). 
\end{align} Also, we denote $\mfw[w_n]$ to be the finite reduced sequence $[w_1,\dots,w_{n-1}]$. If $w_{n+1}\neq w_n$, then we denote $\mfw[w_{n+1}]=[w_1,\dots,w_{n},w_{n+1}]$.
\begin{example}[\emph{Type $B_2$}]\label{ex: GCA}
	Let the initial cluster be $\mfx=(x_1,x_2)$ and the initial triple $(B,\mcZ,R)$ be as follows:
\begin{align}
	B=\begin{pmatrix}0 & -1 \\ 1 & 0 \end{pmatrix}, \left\{
		\begin{array}{ll}
			Z_1(u)=1+u+u^2 \\
			Z_2(u)=1+u 
		\end{array}, \right. R=\begin{pmatrix}2 & 0 \\ 0 & 1\end{pmatrix}.
\end{align}  Note that the product matrix $BR$ is the exchange matrix for a classical cluster algebra of type $B_2$, that is  \begin{align}BR=\begin{pmatrix}0 & -1 \\ 2 & 0\end{pmatrix}.\end{align} We can easily check the relation \eqref{eq: compatibility} holds. Then, all the exchange matrices are same up to a sign in $\{\pm\}$ and all the exchange polynomials are invariant under the mutations. After a direct calculation, we get all the $6$ distinct clusters as follows:
\begin{align*}
\begin{array}{ll}
	\left\{\begin{array}{ll}
			x_{1;0}=x_1 \\
			x_{2;0}=x_2 
		\end{array}, \right. \left\{\begin{array}{ll}
			x_{1;1}=\frac{1+x_2+x_2^2}{x_1} \\
			x_{2;1}=x_2 
		\end{array}, \right. \left\{\begin{array}{ll}
			x_{1;2}=\frac{1+x_2+x_2^2}{x_1} \\
			x_{2;2}=\frac{1+x_2+x_2^2+x_1}{x_1x_2} 
		\end{array}, \right. \\ 
		\left\{\begin{array}{ll}
			x_{1;3}=\frac{1+2x_1+x_1^2+x_1x_2+x_2+x_2^2}{x_1x_2^2} \\
			x_{2;3}=\frac{1+x_2+x_2^2+x_1}{x_1x_2} 
		\end{array}, \right. \left\{\begin{array}{ll}
			x_{1;4}=\frac{1+2x_1+x_1^2+x_1x_2+x_2+x_2^2}{x_1x_2^2} \\
			x_{2;4}=\frac{1+x_1}{x_2} 
		\end{array}, \right. \left\{\begin{array}{ll}
			x_{1;5}=x_1 \\
			x_{2;5}=\frac{1+x_1}{x_2} 
		\end{array}. \right.
\end{array}
\end{align*} For more details about the relation between generalized cluster algebras and classical cluster algebras, we can refer to \cite{Nak15, NR16, Nak21}.
\end{example}
\subsection{Generalized Euclid tree} In this subsection, we define the generalized Euclid tree based on the classical Euclid tree.

\begin{definition}[\emph{$k$-initial triple}]
	Let $k\in \mbR_{\geq 0}$ and $(a,b,c)$ be a triple in $\mbR_{+}^3$. We call it an \emph{$k$-initial triple} if any component is not equal to the sum of others. Namely, it satisfies that $a\neq b+c+k$, $b\neq a+c+k$ and $c\neq a+b+k$.
\end{definition}
Sometimes, for brevity, we collectively call them \emph{initial triples} without ambiguity.
\begin{definition}[\emph{Classical Euclid tree}]\label{def: classical Euclid}
The \emph{classical Euclid tree} $\mcE$ is a $3$-regular connected graph whose vertices are the triples defined by: the $0$-initial triple is $(a,b,c)\in \mbR_{+}^3$ and the recursion formula is defined by the following three mutations:
\begin{align}
\begin{array}{cc}
  \mcM_1 (x,y,z) &= (y + z, y,z)\\
  \mcM_2 (x,y,z) &= (x,x + z,z) \\
  \mcM_3 (x,y,z) &= (x,y,x + y)
 \end{array}\label{eq: classical}
\end{align} And the edges of the graph labelled by $i$ correspond to the mutation $\mcM_{i}$. Note that $\mcM_i$ is not an involution any more as the cluster mutation $\mu_i$.
\end{definition}
It is direct that $\mcE$ is uniquely determined by the initial triple. Later, we will prove that $\mcE$ has a tree structure. In particular, if $(a,b,c)=(1,1,1)$, then $\mcE$ is the most well-known original Euclid tree, which has a deep connection with \emph{Farey triples} and the \emph{Markov type cluster algebra}, see \cite{SV17, Cha13}.
\begin{definition}[\emph{Generalized Euclid tree}]\label{def: GET}
  For any $k\in\mbR_{\geq 0}$, the \emph{$k$-generalized Euclid tree} $\mcK$ with the $k$-initial triple $(A,B,C)\in \mbR_{+}^3$ is defined recursively by the following three mutations:
  \begin{equation}
  \begin{aligned}\label{GMR}
   \mcM_{1;k} (X,Y,Z) &= (k + Y + Z, Y,Z)\\
     \mcM_{2;k}(X,Y,Z) &= (X,k + X + Z,Z) \\
    \mcM_{3;k}(X,Y,Z) &= (X,Y,k + X + Y)
  \end{aligned}
  \end{equation}
  In particular, if $k=0$, then we have $\mcM_{i;0}=\mcM_{i}$ for $i=1,2,3$ and it corresponds to the classical Euclid tree as above.
\end{definition}
Note that each triple $(X,Y,Z)$ in the $k$-generalized Euclid tree with the initial triple $(A,B,C)$ can be written by 
\begin{align}
	(X,Y,Z)=\mcM_{k}^{\mfw}(A,B,C)\coloneq(\mcM_{w_n;k}\circ\dots\circ\mcM_{w_1;k})(A,B,C),
\end{align} 
where $\mfw=[w_1,\dots,w_n]$ is a uniquely determined finite reduced sequence. 
\begin{example}
	We can refer to \Cref{graph: Classical Euclid tree} and \Cref{graph: Generalized Euclid tree} as the examples of classical and $7$-generalized Euclid trees, whose initial triples are $(1,1,1)$ and $(1,4,9)$ respectively. 
\end{example}

\begin{figure}[hbtp]
\centering
\begin{minipage}{0.465\linewidth}
\centering
\scalebox{0.7}{
\begin{tikzpicture}
\node (root) {$(1,1,1)$};
\node[above=of root] (n21) {$(2,1,1)$};
 \draw[->] (root) -- (n21) node[midway, left] {$\mcM_{1}$};
 \node[below left=0.8cm and 0.8cm of root] (n22) {$(1,2,1)$};
  \draw[->] (root) -- (n22) node[midway, left, xshift=6pt, yshift=10pt] {$\mcM_{2}$};
   \node[below right=0.8cm and 0.8cm of root] (n23) {$(1,1,2)$};
    \draw[->] (root) -- (n23) node[midway, left, xshift=20pt, yshift=10pt] {$\mcM_{3}$};
    \node[below = of n23] (n24) {$(3,1,2)$};
     \draw[->] (n23) -- (n24) node[midway, right] {$\mcM_{1}$};
      \node[above right=0.8cm and 0.8cm of n23] (n25) {$(1,3,2)$};
      \draw[->] (n23) -- (n25) node[midway, left, xshift=10pt, yshift=10pt] {$\mcM_{2}$};
\end{tikzpicture}}
\caption{Classical Euclid tree}\label{graph: Classical Euclid tree}
\end{minipage}
\begin{minipage}{0.52\linewidth}
\centering
\scalebox{0.7}{
\begin{tikzpicture}
\node (root) {$(1,4,9)$};
\node[above=of root] (n21) {$(20,4,9)$};
 \draw[->] (root) -- (n21) node[midway, left] {$\mcM_{1;7}$};
 \node[below left=0.8cm and 0.8cm of root] (n22) {$(1,17,9)$};
  \draw[->] (root) -- (n22) node[midway, left, xshift=6pt, yshift=8pt] {$\mcM_{2;7}$};
   \node[below right=0.8cm and 0.8cm of root] (n23) {$(1,4,12)$};
    \draw[->] (root) -- (n23) node[midway, left, xshift=22pt, yshift=10pt] {$\mcM_{3;7}$};
    \node[below = of n23] (n24) {$(23,4,12)$};
     \draw[->] (n23) -- (n24) node[midway, right] {$\mcM_{1;7}$};
      \node[above right=0.8cm and 0.8cm of n23] (n25) {$(1,20,12)$};
      \draw[->] (n23) -- (n25) node[midway, left, xshift=10pt, yshift=10pt] {$\mcM_{2;7}$};
\end{tikzpicture}}
\caption{$7$-generalized Euclid tree}\label{graph: Generalized Euclid tree}
\end{minipage}
\end{figure}

Now,  we show that the generalized Euclid tree is indeed equipped with a tree structure.

\begin{lemma} 
  Let $\mcE$ be the classical Euclid tree with the initial triple $(a,b,c)$ and $\mcK$ be the $k$-generalized Euclid tree with the initial triple $(A,B,C)$. Then, both $\mcE$ and $\mcK$ are $3$-regular trees. Moreover, there is a canonical isomorphism between $\mcE$ and $\mcK$.
\end{lemma}

\begin{proof}
	Without loss of generality, we may only consider the $k$-generalized Euclid tree $\mcK$. Firstly, since $(A,B,C)$ is an initial triple, and by the mutation rules \eqref{GMR}, we conclude that there is no loop in $\mcK$. Suppose that there is a  cycle $\mcO$ in $\mcK$ with length no less than $3$. Note that all the triples in $\mcO$ are different. Then, there exists a unique triple $(A_1,B_1,C_1)$ and two distinct sequences $\mfw$ and $\mfw^{\prime}$, such that 
	\begin{align}
		(A_1,B_1,C_1)=\mcM_{k}^{\mfw}(A,B,C)=\mcM_{k}^{\mfw^{\prime}}(A,B,C),
	\end{align} where $\mfw=[w_1,\dots,w_{n-1},w_n]$ and $\mfw^{\prime}=[w_1^{\prime},\dots,w_{m-1}^{\prime},w_m^{\prime}]$. By the mutation rule \eqref{GMR}, we obtain that $w_n=w_m^{\prime}$ and \begin{align}
		\mcM_{k}^{\mfw[w_n]}(A,B,C)=\mcM_{k}^{\mfw^{\prime}[w_m^{\prime}]}(A,B,C),\end{align} which contradicts with the fact that all the triples in $\mcO$ are different. Hence, there is no cycle in $\mcK$ and we conclude that both $\mcE$ and $\mcK$ are $3$-regular trees.
		
		Now, we define a map $\Phi: \mcE \longrightarrow \mcK$, such that for any finite reduced sequence $\mfw=[r_1,\dots,r_n]$, \begin{align}
			\Phi(\mcM^{\mfw}(a,b,c))=\mcM_{k}^{\mfw}(A,B,C).
		\end{align} Then, it is clear that $\Phi$ is a bijection between the triples in $\mcE$ and $\mcK$ since they are both determined uniquely by $\mfw$. Moreover, let $w_{n+1}\in \{1,2,3\}\backslash\{w_n\}$. Then, we obtain the commutative diagram as follows:
\[
\xymatrix@C=2cm{
 \mcM^{\mfw}(a,b,c) \ar[r]^{\mcM_{w_{n+1}}} \ar[d]_\Phi & \mcM^{\mfw[w_{n+1}]}(a,b,c)  \ar[d]^\Phi  \\
   \mcM_k^{\mfw}(A,B,C) \ar[r]^{\mcM_{w_{n+1};k}} & \mcM_k^{\mfw[w_{n+1}]}(A,B,C) }
\]
Hence, $\phi$ is a canonical isomorphism between 
$\mcE$ and $\mcK$.
\end{proof}

\subsection{Fibonacci sequence} The study of Fibonacci sequence leads to interesting developments in various directions, such as Diophantine approximation, continued fractions. Let us recall its definition and some important properties.
\begin{definition}[\emph{Fibonacci sequence}]\label{def: Fibonacci}
	The Fibonacci sequence $\{F_n\}\ (n\geq 0)$ is defined recursively by: $F_0=0$, $F_1=1$ and for any $n\geq 0$, 
	\begin{align}
		F_{n+2}=F_{n+1}+F_n.
	\end{align} In other words, each term is the sum of the two preceding terms.
\end{definition}
The first few terms are $F_0=0,F_1=1,F_2=1,F_3=2,F_4=3,F_5=5,F_6=8$ and the sequence grows exponentially. In fact, this sequence enjoys some well-known properties as follows.
\begin{lemma}\label{lem: Fibonacci property} The following statements about Fibonacci sequence hold.
	\begin{enumerate}
		\item The Binet's formula holds: \begin{align}
			F_n=\frac{\varphi^n-\psi^n}{\sqrt{5}},\ \varphi=\frac{1+\sqrt{5}}{2}, \psi=\frac{1-\sqrt{5}}{2}. \label{Binet}
		\end{align}
		\item The addition formula holds: $
			F_{n+k}=F_kF_{n+1}+F_{k-1}F_n,\ (n,k\geq 1).$
		\item The Catalan's identity holds: $
			F_n^2-F_{n-r}F_{n+r}=(-1)^{n-r}F_r^2.$
		\item The summation identity holds: $
			F_0+F_1+\dots+F_n=F_{n+2}-1.$
	\end{enumerate}
\end{lemma}
\begin{proposition}\label{prop: Fibonacci converge}
	The series of the sum of reciprocals of the Fibonacci sequence is convergent. That is to say, 
	\begin{align}
		\sum_{n=1}^{\infty}\frac{1}{F_n}<\infty.
	\end{align}
\end{proposition}
\begin{proof}
	Denote $\varphi=\frac{1+\sqrt{5}}{2}$. By the Binet's formula \eqref{Binet}, we have $F_n\geq \varphi^{n-2}$ for $n\geq 2$, which implies that $\frac{1}{F_n}\leq \varphi^{2-n}$. Hence, we obtain that 
	\begin{align}
		\sum_{n=2}^{\infty}\frac{1}{F_n}\leq \sum_{n=2}^{\infty}\varphi^{2-n}=\sum_{m=0}^{\infty}\varphi^{-m}=\frac{1}{1-\varphi^{-1}}<\infty,
	\end{align}  which implies that this series is convergent.
\end{proof}

\section{Fock-Goncharov tropicalization of generalized Markov equations}\label{FG trop}
In this section, we study the Fock-Goncharov tropicalization of generalized Markov equations. To start with, let us remind the relation between the generalized Markov equations and the generalized cluster algebras.
\subsection{Generalized Markov equation}
In \cite{GM23}, Gyoda and Matsushita introduced the generalized Markov equation, which is a generalization of the classical Markov equation.
The generalized Markov equations are given as follows:
\begin{align}
  X_1^2 + X_2^2 + X_3^2 + \lambda_3 X_1 X_2 + \lambda_1 X_2 X_3 + \lambda_2 X_3 X_1 = (3+ \lambda_1 + \lambda_2 + \lambda_3) X_1 X_2 X_3, \label{eq: GME}
\end{align}
where $\lambda_1, \lambda_2, \lambda_3 \in \mbZ_{\geq 0}$. In particular, if $\lambda_1=\lambda_2=\lambda_3=0$, it is the classical Markov equation.

Firstly, we recall the structure of generalized cluster algebra behind. Let the initial triple $(B,\mcZ,R)$ be as follows:
\begin{align}
	B=\begin{pmatrix}0 & 1 & -1\\ -1 & 0 & 1\\ 1 & -1 & 0\end{pmatrix}, \left\{
		\begin{array}{ll}
			Z_1(u)=1+\lambda_1u+u^2 \\
			Z_2(u)=1+\lambda_2u+u^2\\ 
			Z_3(u)=1+\lambda_3u+u^2 
		\end{array}, \right. R=\begin{pmatrix}2 & 0 & 0\\ 0 & 2 & 0\\ 0 & 0 & 2\end{pmatrix}.
\end{align}
Then, the mutation rules are given by
\begin{align}
\begin{array}{cc}
    \mu_1(x_1, x_2, x_3) &= (\dfrac{x_2^2 + \lambda_1 x_2 x_3 + x_3^2}{x_1}, x_2, x_3)\\
    \mu_2(x_1, x_2, x_3) &= (x_1, \dfrac{x_1^2 + \lambda_2 x_1 x_3 + x_3^2}{x_2}, x_3)\\
    \mu_3(x_1, x_2, x_3) &= (x_1, x_2, \dfrac{x_1^2 + \lambda_3 x_1 x_2 + x_2^2}{x_3})
\end{array}.\label{mutation rule}
\end{align}
In fact, these mutation rules can be naturally regarded as the maps $\mu_i: \mbQ^3_{+} \to \mbQ^3_{+}$.
Recall that for any finite mutation sequence $\mfw=[w_1,w_2,\dots,w_n]\in \mcT$, we denote 
\begin{align}
	\mu^{\mfw}(x_1,x_2,x_3)=(\mu_{w_n}\circ\dots\circ\mu_{w_1})(x_1,x_2,x_3).
\end{align} 
\begin{lemma}[{\cite[Lemma 4 \& Corollary 6]{GM23}}]\label{lem: monotonicity} The following statements hold.
	\begin{enumerate}
		\item The only solutions to \eqref{eq: GME} that contain repeated numbers are $(1,1,1)$, $(k_1+2,1,1)$, $(1,k_2+2,1)$ and $(1,1,k_3+2)$, which are said to be singular.
		\item Let $(a,b,c)\neq (1,1,1)$ be a  positive integer to \eqref{eq: GME}. Set $(a^{\prime},b,c)=\mu_1(a,b,c)$, $(a,b^{\prime},c)=\mu_2(a,b,c)$ and $(a,b,c^{\prime})=\mu_3(a,b,c)$. 
		\begin{itemize}[leftmargin=0em]\itemsep=0pt \item If $a=\max(a,b,c)$, then $a^{\prime}\neq \max(a,b,c)$, $b^{\prime}=\max(a,b^{\prime},c)$ and $c^{\prime}=\max(a,b,c^{\prime})$.
		\item If $b=\max(a,b,c)$, then $a^{\prime}= \max(a,b,c)$, $b^{\prime}\neq \max(a,b^{\prime},c)$ and $c^{\prime}=\max(a,b,c^{\prime})$.
		\item If $c=\max(a,b,c)$, then $a^{\prime}= \max(a,b,c)$, $b^{\prime}=\max(a,b^{\prime},c)$ and $c^{\prime}\neq \max(a,b,c^{\prime})$.

		\end{itemize}
	\end{enumerate}
\end{lemma}
\begin{theorem}[{\cite[Theorem 1]{GM23}}]\label{thm: generate}
	Every positive integer solution $(a,b,c)$ to \eqref{eq: GME} can be generated by the initial solution $(1,1,1)$ by finitely many mutations. That is to say, there exists $\mfw \in \mcT$ with $|\mfw|< \infty$, such that 
	\begin{align}
		(a,b,c)=\mu^{\mfw}(1,1,1).
	\end{align}
\end{theorem}
For example, if we take $\lambda_1=0$ and $\lambda_2=\lambda_3=2$, the generating rules of the generalized Markov triples can be referred to \Cref{tree: generalized Markov triples}.
\begin{figure}[htbp]
\begin{align}
\scalebox{0.85}{
\begin{xy}(0,0)*+{(1,1,1)}="0",(20,20)*+{(2,1,1)}="1",(20,0)*+{(1,4,1)}="1'",(20,-20)*+{(1,1,4)}="1''",(45,50)*+{(2,9,1)}="20",(45,30)*+{(2,1,9)}="21",(45,10)*+{(17,4,1)}="22",(45,-10)*+{(1,4,25)}="23",(45,-30)*+{(17,1,4)}="24",(45,-50)*+{(1,25,4)}="25",(80,55)*+{(41,9,1)\cdots}="40",(80,45)*+{(2,9,121)\cdots}="41", (80,35)*+{(41,1,9)\cdots}="42", (80,25)*+{(2,121,9)\cdots}="43", (80,15)*+{(17,81,1)\cdots}="44", (80,5)*+{(17,4,441)\cdots}="45", (80,-5)*+{(641,4,25)\cdots}="46", (80,-15)*+{(1,169,25)\cdots}="47", (80,-25)*+{(17,441,4)\cdots}="48", (80,-35)*+{(17,1,81)\cdots}="49", (80,-45)*+{(641,25,4)\cdots}="410", (80,-55)*+{(1,25,169)\cdots}="411", \ar@{-}^{\mu_1}"0";"1"\ar@{-}^{\mu_2}"0";"1'"\ar@{-}_{\mu_3}"0";"1''"\ar@{-}^{\mu_2}"1";"20"\ar@{-}_{\mu_3}"1";"21"\ar@{-}^{\mu_1}"1'";"22"\ar@{-}_{\mu_3}"1'";"23"\ar@{-}^{\mu_1}"1''";"24"\ar@{-}_{\mu_2}"1''";"25"\ar@{-}^{\mu_1}"20";"40"\ar@{-}_{\mu_3}"20";"41"\ar@{-}^{\mu_1}"21";"42"\ar@{-}_{\mu_2}"21";"43"\ar@{-}^{\mu_2}"22";"44"\ar@{-}_{\mu_3}"22";"45"\ar@{-}^{\mu_1}"23";"46"\ar@{-}_{\mu_2}"23";"47"\ar@{-}^{\mu_2}"24";"48"\ar@{-}_{\mu_3}"24";"49"\ar@{-}^{\mu_1}"25";"410"\ar@{-}_{\mu_3}"25";"411"
\end{xy}}\notag
\end{align}
\caption{generalized Markov triples for $\lambda_1=0,\lambda_2=\lambda_3=2$}
\label{tree: generalized Markov triples}
\end{figure}
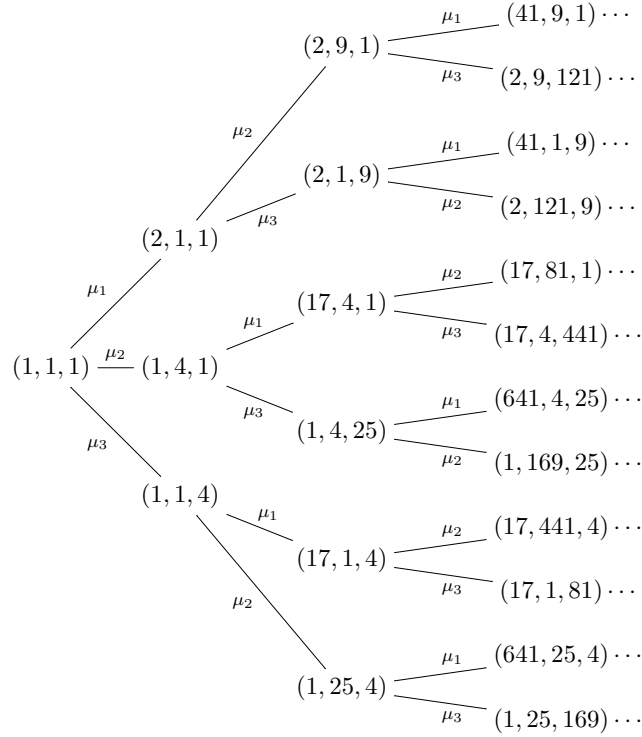
\subsection{Deformed Fock-Goncharov tropicalization}
Now, we briefly recall some basic notions about positive space and tropicalization from \cite{FG09, She14}.

A \emph{positive space} is a variety $\mcA$ equipped with a positive atlas $\mcC_{\mcA}$. That is to say, the transition maps between the coordinate systems in $\mcC_{\mcA}$ are rational functions, which are in the form of a ratio of two polynomials with positive integral coefficients. Such positive space is denoted by $(\mcA, \mcC_{\mcA})$.

Then, we recall that a \emph{universally positive Laurent polynomial} on $\mcA$ is a regular function on $\mcA$, such that it is always a Laurent polynomial with nonnegative integral coefficients in every coordinate system in $\mcC_{\mcA}$. We denote by $\mbL_+(\mcA)$ the set of all the universally positive Laurent polynomials on $\mcA$.
\begin{definition}[\emph{Semifield}]
	A \emph{semifield} is a multiplicative abelian group $(\mbP,\cdot)$ which is equipped with an additive operation $\oplus$ such that for any $a,b,c \in \mbP$,
\begin{enumerate}[leftmargin=2em]
	\item $a \oplus b=b \oplus a$,
	\item $(a \oplus b)c = ac \oplus bc$,
	\item $(a \oplus b) \oplus c = a \oplus (b \oplus c).$
\end{enumerate}
Then, we denote it by $(\mbP, \cdot, \oplus)$.
\end{definition}
\begin{example}\label{ex: semifield} The following are two important examples of semifields.
\begin{enumerate}[leftmargin=2em]
	\item Let $\mbA^t$ be a set of numbers, where $\mbA=\mbZ,\mbQ,\mbR$. Then, it becomes a semifield if the multiplication and addition are given by 
	\begin{align}
		a \cdot b \coloneq a+b,\ a\oplus b \coloneq \max\{a,b\}.
	\end{align}
	\item  Let $\mbP_{\text{trop}}=\text{Trop}(u_1,\dots, u_n)$ be a multiplicative abelian group freely generated by formal variables $u_1,\dots,u_n$ with addition $\oplus$ as follows:
\begin{align}
\prod_{i=1}^n u_i^{a_i} \oplus \prod_{i=1}^{n} u_i^{b_i}\coloneq\prod_{i=1}^{n} u_j^{\min(a_i,b_i)}.
\end{align} Then, it is a semifield, which is called the \emph{tropical semifield}.
\end{enumerate}
\end{example} 
In fact, for any positive space $\mcA$, the transition maps are subtraction-free. Hence, we can take any semifield $\mbP$ and consider the set $\mcA(\mbP)$ of $\mbP$-points of $\mcA$. Note that $\mcA(\mbP) \simeq \mbP^n$, see \cite{She14}. Let $\mcA(\mbZ^t)$ be the set of $\mbZ^t$-points. Then, we can tropicalize $F\in \mbL_+(\mcA)$ by evaluating it on $\mcA(\mbZ^t)$, denoted by $F^t$. It can be checked directly that $F^t$ is a convex piecewise linear function in each positive coordinate system. 
\begin{example}
	Let $F=(X_1^2+X_2^2+X_3^2+\lambda_3X_1X_2+\lambda_1X_2X_3+\lambda_2X_3X_1)(X_1X_2X_3)^{-1}$ be a positive integral Laurent polynomial, where $\lambda_i\in \mbN_+$. Then, we consider its tropicalization $F^t$ in $\mcA(\mbZ^t)$. When taking the maximum, we can drop the coefficients of monomials in $F$ since they do not matter. Hence, we obtain that
	\begin{align}
		\begin{array}{ll}	F^t&=\max(2X_1,2X_2,2X_3,X_1+X_2,X_2+X_3,X_3+X_1)-(X_1+X_2+X_3)\\ &=\max(2X_1,2X_2,2X_3)-(X_1+X_2+X_3).
		\end{array}
	\end{align}
\end{example}
A more explicit expression of the tropicalization can be obtained as follows.
\begin{proposition}[cf. {\cite{FG09}}] \label{prop: trop}
	The positive integral Laurent polynomial $F$ and its tropicalization $F^t$ are related as follows:
	\begin{align}
		\lim_{C\rightarrow \infty}\dfrac{\log F(e^{CX_1},\dots,e^{CX_n})}{C}=F^t(X_1,\dots,X_n),\ X_i\in \mbZ.
	\end{align}
\end{proposition}
\begin{proof}
	We only need to note that 
	\begin{align}
		\lim_{C\rightarrow \infty}\dfrac{\log (e^{CX_1}+\dots+e^{CX_n})}{C}=\max(X_1,\dots,X_n).	\end{align}
	Since the tropicalization respectively transfers the ordinary multiplication and addition to the addition and maximum. Hence, this implies that the proposition holds.
\end{proof} 
Here, we also call this process the \emph{Fock-Goncharov tropicalization}. However, for our purpose,  we need to introduce a novel notion based on this, which is called \emph{deformed Fock-Goncharov tropicalization}. That is to say, given a positive integral Laurent polynomial $F(X_1,\dots,X_n)$, we firstly get $F^t(X_1,\dots,X_n)\in \mcA(\mbZ^t)$ under the Fock-Goncharov tropicalization. Secondly, we replace all the $X_i$ in $F^t(X_1,\dots,X_n)$ by the $\mbZ$-valued variables $x_i$, such that 
\begin{enumerate}[leftmargin=2em]
	\item If the positive integral Laurent polynomial satisfies  $F(X_1,\dots,X_n)=k$ for some $k\in \mbN$, then $F^t(x_1,\dots,x_t)=0$. 
	\item If $X_i=\max(X_1,\dots,X_n)$, then $x_i=\max(x_1,\dots,x_n)$. If $X_i\neq \max(X_1,\dots,X_n)$, then $x_i\neq \max(x_1,\dots,x_n)$.
\end{enumerate}
Now, we can tropicalize the generalized Markov tree. We may prove that it  essentially has a classical Euclid tree structure under the deformed Fock-Goncharov tropicalization.
\begin{theorem}\label{thm: FG trop}
	The deformed Fock-Goncharov tropicalization of the generalized Markov tree is the classical Euclid tree.
\end{theorem}
\begin{proof}
	Recall that the generalized Markov equation is 
	\begin{align}
		X_1^2 + X_2^2 + X_3^2 + \lambda_3 X_1 X_2 + \lambda_1 X_2 X_3 + \lambda_2 X_3 X_1 = (3+ \lambda_1 + \lambda_2 + \lambda_3) X_1 X_2 X_3. \label{eq: Markov 2}
	\end{align} Then, by a direct calculation, we can get the  \emph{deformed Fock-Goncharov tropicalized} generalized Markov equation equation as follows:
	\begin{align}
\max(2x_1,2x_2,2x_3)=x_1+x_2+x_3.\label{eq: trop markov}
	\end{align} Note that the mutation rules \eqref{mutation rule} of the generalized Markov cluster algebras are Laurent polynomials with positive integer coefficients. Hence, by \Cref{prop: trop}, the tropicalized mutation rules are given by 
	\begin{align}
		\begin{array}{cc}
    \mu_1^t(x_1, x_2, x_3) &= (\max(2x_2,2x_3)-x_1, x_2, x_3)\\
    \mu_2^t(x_1, x_2, x_3) &= (x_1, \max(2x_1,2x_3)-x_2, x_3)\\
    \mu_3^t(x_1, x_2, x_3) &= (x_1, x_2, \max(2x_1,2x_2)-x_3)
\end{array}.
	\end{align} If $(X_1,X_2,X_3)=(1,1,1)$, we can technically take $(x_1,x_2,x_3)=(0,0,0)$, which is also said to be \emph{singular}. For convenience, we assume that $(X_1,X_2,X_3)\neq (1,1,1)$ is any  positive integer solution to \eqref{eq: Markov 2}.  According to \Cref{thm: generate}, suppose that one of the next two solutions is $(X_1^{\prime},X_2,X_3)=\mu_1(X_1,X_2,X_3)$, that is to say $X_1^{\prime}=\max(X_1^{\prime},X_2,X_3)$. Then, by \Cref{lem: monotonicity}, we have $X_1\neq \max(X_1,X_2,X_3)$. Hence, it implies that its deformed Fock-Goncharov tropicalization satisfies $x_1\neq \max(x_1,x_2,x_3)$, and we have 
	\begin{align}
		\max(2x_2,2x_3)-x_1=\max(2x_1,2x_2,2x_3)-x_1=x_2+x_3.
	\end{align} Thus, this implies that $\mu_1^t(x_1,x_2,x_3)=(\max(2x_2,2x_3)-x_1, x_2, x_3)=(x_2+x_3,x_2,x_3)$. Similarly, if $x_2\neq \max(x_1,x_2,x_3)$ and $x_3\neq \max(x_1,x_2,x_3)$, then we have $\mu_2^t(x_1,x_2,x_3)=(x_1,x_1+x_3,x_3)$ and $\mu_3^t(x_1,x_2,x_3)=(x_1,x_2,x_1+x_2)$. Therefore, these tropicalized mutation rules exactly correspond to the mutation rules \eqref{eq: classical} of the classical Euclid tree.
\end{proof}
\begin{example}\label{ex: deform}
	Consider the well-known Markov equation for $\lambda_1=\lambda_2=\lambda_3=0$, that is \begin{align}X_1^2+X_2^2+X_3^2=3X_1X_2X_3.\label{eq: Markov}
	\end{align} Then, its deformed Fock-Goncharov tropicalization is also given by 
	\begin{align}
		\max(2x_1,2x_2,2x_3)=x_1+x_2+x_3.
	\end{align} Take the positive integer solutions $(X_1,X_2,X_3)=(2,1,1)$ and $(x_1,x_2,x_3)=(4,2,2)$. Note that the choice of $(x_1,x_2,x_3)$ is not unique. (We can also choose $(5,3,2)$.) It holds that $X_1=\max(X_1,X_2,X_3)$ and $x_1=\max(x_1,x_2,x_3)$. Take the mutation sequence $\mfw=[2,3,1]$. Then, we compare two sequences as follows:
	\begin{align}
		\begin{array}{ll}
		\mathrm{Markov:}\ 	(2,1,1)\stackrel{\mu_2}{\longrightarrow} (2,5,1)\stackrel{\mu_3}{\longrightarrow} (2,5,29) \stackrel{\mu_1}{\longrightarrow} (433,5,29)\\
		 \mathrm{Tropicalization:}\ (4,2,2)\stackrel{\mu_2^t}{\longrightarrow} (4,6,2)\stackrel{\mu_3^t}{\longrightarrow} (4,6,10) \stackrel{\mu_1^t}{\longrightarrow} (16,6,10)	
		\end{array} .
	\end{align} It can be seen directly that the positions of the maximum components are in one-to-one correspondence. Moreover, the tropicalized chain shapes the same as the classical Euclid tree. 
\end{example}
\begin{remark}
	Although there is a good generating relation between the deformed Fock-Goncharov tropicalized generalized Markov tree and the classical Euclid tree, more concrete relation between the values of generalized Markov tree and the classical Euclid tree is still mysterious. Hence, a natural question arises as follows.
\end{remark}
\begin{question} 
Is there a quantitative relationship between the generalized Markov tree and the classical Euclid tree? In the next several sections, we aim to solve this question and show their explicit relation. 
\end{question}
\section{Comparison between generalized Euclid tree and classical Euclid tree}\label{comparison}

In this section, we will compare the $k$-generalized Euclid tree $\mcK$ and the classical Euclid tree $\mcE$. We show that they are essentially same up to a scalar multiple at infinity.

\begin{definition}[\emph{Comparison triple}]\label{def: Comparison triple}
	Under the bijection $\Phi$ between the triples $(x_i,y_i,z_i)\in \mcE$ and $(X_i,Y_i,Z_i)\in \mcK$, the \emph{comparison triple} is defined by 
	\begin{align}
		(l_i, m_i, n_i) \coloneq (\frac{X_i}{x_i},\frac{Y_i}{y_i},\frac{Z_i}{z_i}).
	\end{align}
\end{definition}
Hence, the mutation of the comparison triples is induced by the mutation of $k$-generalized Euclid tree and classical Euclid tree. It is complicated but has many interesting properties.

Now, consider any simple mutation chain in the $k$-generalized Euclid tree $\mcK$. For example, if we take \begin{align}(X_0,Y_0,Z_0)\xrightarrow{\mcM_{2;k}} (X_1,Y_1,Z_1),\end{align}
then the corresponding simple mutation chain in classical Euclid tree $\mcE$ is denoted by \begin{align}(x_0,y_0,z_0)\xrightarrow{\mcM_2} (x_1,y_1,z_1).\end{align}
Hence, we have $X_1=X_0$, $Y_1= X_0+ Z_0+ k$, $Z_1=Z_0$, and $x_1=x_0$, $y_1 = x_0 + z_0$, $z_1=z_0$. It implies that $l_1= l_0 = \frac{X_1}{x_1}$, $n_1 = n_0 = \frac{Z_1}{z_1}$, and the number $m_1$ can be expressed by 
\begin{align}
\begin{array}{ll}
  m_1 &= \frac{Y_1}{y_1} = \frac{X_0+ Z_0+ k}{x_0 + z_0} = \frac{X_0}{x_0 + z_0} + \frac{Z_0}{x_0 + z_0} + \frac{k}{x_0 + z_0}\\ &= \frac{x_0}{x_0 + z_0}\times l_0 + \frac{z_0}{x_0 + z_0}\times n_0 + \frac{k}{x_0 + z_0} \\
  &= \frac{x_0}{x_0 + z_0}\times (l_0 + \frac{k}{x_0 + z_0}) + \frac{z_0}{x_0 + z_0}\times (n_0 + \frac{k}{x_0 + z_0}).
  \end{array}
\end{align}
Thus, the corresponding simple mutation chain of the comparison triple can be written as $(l_0, m_0, n_0) \xrightarrow{\delta_2} (l_1, m_1, n_1)$, where $l_1=l_0$, $n_1=n_0$ and $m_1 = \frac{x_0}{x_0 + z_0}\times (l_0 + \frac{k}{x_0 + z_0}) + \frac{z_0}{x_0 + z_0}\times (n_0 + \frac{k}{x_0 + z_0})$.
Note that we can similarly calculate the other two mutations $\delta_1$, $\delta_3$ and respectively get the changed elements in the comparison triples as follows:
\begin{align}
	\begin{array}{ll}
		l_1=\frac{y_0}{y_0 + z_0}\times (m_0 + \frac{k}{y_0 + z_0}) + \frac{z_0}{y_0 + z_0}\times (n_0 + \frac{k}{y_0 + z_0}),\\
		n_1=\frac{x_0}{x_0 + y_0}\times (l_0 + \frac{k}{x_0 + y_0}) + \frac{y_0}{x_0 + y_0}\times (m_0 + \frac{k}{x_0 + y_0}).
	\end{array}
\end{align}
More precisely, we can illustrate the mutation of comparison triple $\delta_2$ (or $\delta_1,\delta_3$) as the following \Cref{cluster real}, where $\Delta = \frac{k}{x_0 + z_0}$. Without loss of generality, we may assume that $l_0<n_0$.

\begin{figure}[htbp]
\begin{flushleft}
\tikzset{every picture/.style={line width=0.75pt}} 
\begin{tikzpicture}[x=0.75pt,y=0.75pt,yscale=-1,xscale=1]
\path (0,100); 

\draw    (170,79.83) -- (439.67,80.5) ;
\draw  [color={rgb, 255:red, 208; green, 2; blue, 27 }  ,draw opacity=1 ][fill={rgb, 255:red, 208; green, 2; blue, 27 }  ,fill opacity=1 ] (227.33,80) .. controls (227.33,78.53) and (228.53,77.33) .. (230,77.33) .. controls (231.47,77.33) and (232.67,78.53) .. (232.67,80) .. controls (232.67,81.47) and (231.47,82.67) .. (230,82.67) .. controls (228.53,82.67) and (227.33,81.47) .. (227.33,80) -- cycle ;
\draw  [color={rgb, 255:red, 208; green, 2; blue, 27 }  ,draw opacity=1 ][fill={rgb, 255:red, 208; green, 2; blue, 27 }  ,fill opacity=1 ] (377.33,80) .. controls (377.33,78.53) and (378.53,77.33) .. (380,77.33) .. controls (381.47,77.33) and (382.67,78.53) .. (382.67,80) .. controls (382.67,81.47) and (381.47,82.67) .. (380,82.67) .. controls (378.53,82.67) and (377.33,81.47) .. (377.33,80) -- cycle ;
\draw    (170,165.5) -- (439.67,166.17) ;
\draw  [color={rgb, 255:red, 208; green, 2; blue, 27 }  ,draw opacity=1 ][fill={rgb, 255:red, 208; green, 2; blue, 27 }  ,fill opacity=1 ] (247.67,165.67) .. controls (247.67,164.19) and (248.86,163) .. (250.33,163) .. controls (251.81,163) and (253,164.19) .. (253,165.67) .. controls (253,167.14) and (251.81,168.33) .. (250.33,168.33) .. controls (248.86,168.33) and (247.67,167.14) .. (247.67,165.67) -- cycle ;
\draw  [color={rgb, 255:red, 208; green, 2; blue, 27 }  ,draw opacity=1 ][fill={rgb, 255:red, 208; green, 2; blue, 27 }  ,fill opacity=1 ] (397.67,166) .. controls (397.67,164.53) and (398.86,163.33) .. (400.33,163.33) .. controls (401.81,163.33) and (403,164.53) .. (403,166) .. controls (403,167.47) and (401.81,168.67) .. (400.33,168.67) .. controls (398.86,168.67) and (397.67,167.47) .. (397.67,166) -- cycle ;
\draw  [dash pattern={on 0.84pt off 2.51pt}]  (230,80) -- (229.83,165.58) ;
\draw  [dash pattern={on 0.84pt off 2.51pt}]  (380,80) -- (379.83,165.58) ;
\draw  [color={rgb, 255:red, 189; green, 16; blue, 224 }  ,draw opacity=1 ][fill={rgb, 255:red, 189; green, 16; blue, 224 }  ,fill opacity=1 ] (337.33,165.67) .. controls (337.33,164.19) and (338.53,163) .. (340,163) .. controls (341.47,163) and (342.67,164.19) .. (342.67,165.67) .. controls (342.67,167.14) and (341.47,168.33) .. (340,168.33) .. controls (338.53,168.33) and (337.33,167.14) .. (337.33,165.67) -- cycle ;
\draw    (301.64,104.7) -- (301.78,125.07)(298.64,104.72) -- (298.78,125.09) ;
\draw [shift={(300.33,132.08)}, rotate = 269.6] [color={rgb, 255:red, 0; green, 0; blue, 0 }  ][line width=0.75]    (10.93,-4.9) .. controls (6.95,-2.3) and (3.31,-0.67) .. (0,0) .. controls (3.31,0.67) and (6.95,2.3) .. (10.93,4.9)   ;
\draw    (250.33,165.67) .. controls (285,150.43) and (299.86,151.57) .. (340,165.67) ;
\draw    (340,165.67) .. controls (372.43,149.57) and (374.14,155.29) .. (400.33,166) ;
\draw [line width=2.25]    (229.9,162.74) -- (229.83,168.88) ;
\draw [line width=2.25]    (380.27,163.28) -- (380.2,169.42) ;

\draw (223,172) node [anchor=north west][inner sep=0.75pt]   [align=left] {$l_1$};
\draw (374,172) node [anchor=north west][inner sep=0.75pt]   [align=left] {$n_1$};
\draw (223,60) node [anchor=north west][inner sep=0.75pt]   [align=left] {$l_0$};
\draw (374,63) node [anchor=north west][inner sep=0.75pt]   [align=left] {$n_0$};
\draw (240,140) node [anchor=north west][inner sep=0.75pt]   [align=left] {$l_1+ \Delta$};
\draw (390,140) node [anchor=north west][inner sep=0.75pt]   [align=left] {$n_1 + \Delta$};
\draw (330.5,145) node [anchor=north west][inner sep=0.75pt]   [align=left] {$m_1$};
\draw (305.71,107.57) node [anchor=north west][inner sep=0.75pt]   [align=left] {$\delta_2$};

\end{tikzpicture}
\end{flushleft}
\caption{Mutation of comparision triple at $\delta_2$}
\label{cluster real}
\end{figure}
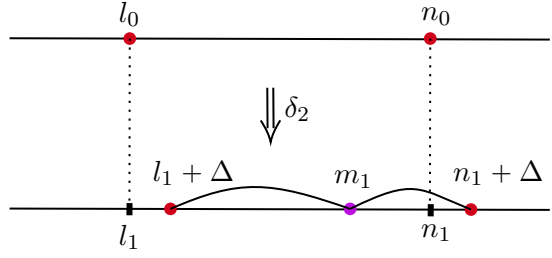
As the above figure shows, if  the length of the mutation sequence is large enough, that is $x_0$ and $z_0$ are sufficiently large, then we can ignore the term $\Delta=\frac{k}{x_0 + z_0}$. Then, the mutation $\delta_2$ approximately becomes $(l_0, m_0, n_0) \xrightarrow{\delta_2} (l_1, m_1, n_1)$, 
where $l_1=l_0$, $n_1=n_0$ and \begin{align}m_1 = \frac{x_0}{x_0 + z_0}\times l_0 + \frac{z_0}{x_0 + z_0}\times n_0.\end{align}
Therefore, approximately, $m_1$ is the internal division point of the other two unmutated points $l_1, n_1$.
This observation leads to several propositions about the asymptotic phenomenon of the $k$-generalized Euclid tree, which we will present in the next section.

Beforehand, we need a preliminary lemma as follows.
\begin{lemma}\label{lem: internal condition} Let $(l_j,m_j,n_j)$ be a comparison triple associated with $(x_j,y_j,z_j)\in \mcE$ and $(X_j,Y_j,Z_j)\in \mcK$. Then, the following statements hold:
	\begin{enumerate}
    \item If $m_j < n_j$ and we mutate at $\delta_1$, then $l_{j+1} \in [m_j,n_j] \Longleftrightarrow k \leq y_j(n_j-m_j)$.
    \item If $n_j < m_j$ and we mutate at $\delta_1$, then $l_{j+1} \in [n_j,m_j] \Longleftrightarrow k \leq z_j(m_j-n_j)$.
    \item If $l_j < n_j$ and we mutate at $\delta_2$, then $m_{j+1} \in [l_j,n_j] \Longleftrightarrow k \leq x_j(n_j-l_j)$.
    \item If $n_j < l_j$ and we mutate at $\delta_2$, then $m_{j+1} \in [n_j,l_j] \Longleftrightarrow k \leq z_j(l_j-n_j)$.
    \item If $l_j < m_j$ and we mutate at $\delta_3$, then $n_{j+1} \in [l_j,m_j] \Longleftrightarrow k \leq x_j(m_j-l_j)$.
    \item If $m_j < l_j$ and we mutate at $\delta_3$, then $n_{j+1} \in [m_j,l_j] \Longleftrightarrow k \leq y_j(l_j-m_j)$.
  \end{enumerate}
 Moreover, if the inequality for $k$ is strict, then the mutated point will lie in the corresponding open interval.
\end{lemma}
\begin{proof}
	 Consider any mutation in $\mcK$: $(X_j,Y_j,Z_j)\xrightarrow{\mcM_{i;k}} (X_{j+1},Y_{j+1},Z_{j+1})$, and the corresponding mutation in $\mcE$: $(x_j,y_j,z_j)\xrightarrow{\mcM_i} (x_{j+1},y_{j+1},z_{j+1})$.
  The mutation of comparison triple  $(l_j, m_j, n_j) \xrightarrow{\delta_i} (l_{j+1}, m_{j+1}, n_{j+1})$ for any $i=1,2,3$ has the form 
 \begin{align} 
 \left\{
    \begin{array}{ccc}
    \delta_1(l_j,m_j,n_j) &=& (\frac{y_j}{y_j + z_j}\times m_j + \frac{z_j}{y_j + z_j}\times n_j + \frac{k}{y_j + z_j}, m_j, n_j) \\
    \delta_2(l_j,m_j,n_j) &=& (l_j, \frac{x_j}{x_j + z_j}\times l_j + \frac{z_j}{x_j + z_j}\times n_j + \frac{k}{x_j + z_j}, n_j) \\
    \delta_3(l_j,m_j,n_j) &=& (l_j, m_j, \frac{x_j}{x_j + y_j}\times l_j + \frac{y_j}{x_j + y_j}\times m_j + \frac{k}{x_j + y_j})    \end{array}. \right.
   \end{align}
  In the case of $i=1$, we have 
  \begin{align}
  \begin{array}{ll}
     [(\frac{y_j}{y_j + z_j}\times m_j + \frac{z_j}{y_j + z_j}\times n_j + \frac{k}{y_j + z_j}) - m_j]\times [(\frac{y_j}{y_j + z_j}\times m_j + \frac{z_j}{y_j + z_j}\times n_j + \frac{k}{y_j + z_j}) - n_j] \\
     = [\frac{z_j}{y_j + z_j}\times (n_j-m_j) + \frac{k}{y_j + z_j}] \times [\frac{y_j}{y_j + z_j}\times (m_j-n_j) + \frac{k}{y_j + z_j}].
    \end{array}\label{eq: internal}
  \end{align}
 Therefore, to decide whether the mutated number $l_{j+1}$ is the internal division point of the other two  unmutated points $m_{j+1}=m_j$ and $n_{j+1}=n_j$, we need to check whether the above product \eqref{eq: internal} is positive or negative.
  This is equivalent to check the product 
 \begin{align}
    [z_j\times (n_j-m_j) + k] \times [y_j\times (m_j-n_j) + k]
 \end{align}
  is positive or negative. Hence, if $m_j<n_j$, then we have \begin{align}l_{j+1} \in [m_j,n_j] \Longleftrightarrow k \leq y_j(n_j-m_j).\end{align} Similarly, we can check the case $n_j<m_j$ and the other two cases of $i=2$ and $i=3$. Hence, we can list all the conditions which allow the mutated one to be the internal division point of the other two unmutated points. Moreover, it is direct that if the equality is excluded, then the mutated point will lie in the open interval.
\end{proof}
\section{Asymptotic phenomenon of the generalized Euclid tree}\label{asym of GET}

In this section, we aim to exhibit the asymptotic phenomenon between the $k$-generalized Euclid tree $\mcK$ and the classical Euclid tree $\mcE$ with the help of Fibonacci sequence.
\begin{definition}[\emph{$3$-cyclic sequence}] \label{def: cyclic seq}
	Let $\mfw=[w_1,w_2,w_3,\dots]$ be an infinite reduced mutation sequence, where the numbers $1,2,3$ appear alternately and any set of three consecutive numbers is $\{1,2,3\}$. Then, $\mfw$ is called a \emph{$3$-cyclic sequence}. \end{definition}
	There are $6$ possible $3$-cyclic sequences which correspond to elements in the symmetry group $\mathfrak{S}_3$. For example, $\mfw_1=[1,2,3,1,2,3,\dots]$ and $\mfw_2=[3,1,2,3,1,2,\dots]$ are both $3$-cyclic sequences. Denote by $(x_i,y_i,z_i)$ the classical Euclid triples along any reduced mutation sequence $\mfw$. Then, we have the following lemma.
\begin{lemma}\label{lem:fibonacci-estimate}
Let $\mfw=[w_1,\dots,w_n,\dots]$ be an infinite mutation sequence with a $3$-cyclic subsequence indexed by $s_i\ (i=1,2,\dots)$.
  If we take the corresponding mutated components of the classical Euclid triples, then they are  bounded below by the Fibonacci sequence $F_i\ (i=1,2,\dots)$. For example, if the mutated elements are  $x_{s_1}, y_{s_2}, z_{s_3}, x_{s_4},y_{s_5}, \cdots $ , then we have
  \begin{align}
    F_1 < x_{s_1},\ F_2 < y_{s_2},\ F_3 < z_{s_3},\ F_4 < x_{s_4},\ F_5 < y_{s_5},\ \cdots.
  \end{align}
\end{lemma}

\begin{proof}
  Inductively, it can be proved directly by the definition of Fibonacci sequence and mutation rules \eqref{eq: classical} of the classical Euclid tree $\mcE$.
\end{proof}

\begin{proposition}\label{prop:comparison-bounded}
  Take any $k$-generalized Euclid tree $\mcK$ with the initial triple $(A,B,C)$ and the classical Euclid tree $\mcE$ with the initial triple $(a,b,c)$. Let $\mfw=[w_1,
  \dots,w_n,\dots]$ be an infinite reduced mutation sequence.
  Assume that $1,2,3$ all appear infinitely many times in $\mfw$.
  Then, the corresponding sequence $\{\max(l_j,m_j,n_j)\}_{j=1}^{+\infty}$ associated with the comparison triples is bounded above.
\end{proposition}

\begin{proof} Suppose that $(X_0,Y_0,Z_0)=(A,B,C)$, $(x_0,y_0,z_0)=(a,b,c)$ and the corresponding comparison triple is $(l_0,m_0,n_0)$. Let $\mfw_j$ be the subsequence of $\mfw$ and denote by $(X_j,Y_j,Z_j)=\mcM^{\mfw_j}_k(A,B,C)$, $(x_j,y_j,z_j)=\mcM^{\mfw_j}(a,b,c)$.

   Recall that for any mutation in $\mcK$: $(X_j,Y_j,Z_j)\xrightarrow{\mcM_{w_{j+1};k}} (X_{j+1},Y_{j+1},Z_{j+1})$ and the corresponding mutation in $\mcE$: $(x_j,y_j,z_j)\xrightarrow{\mcM_{w_{j+1}}} (x_{j+1},y_{j+1},z_{j+1})$, 
  the mutation of the comparison triple is denoted by $(l_j, m_j, n_j) \xrightarrow{\delta_{w_{j+1}}} (l_{j+1}, m_{j+1}, n_{j+1})$, where $w_{j+1}\in \{1,2,3\}$. Now, we start to study each sequence of comparison numbers $\{l_j\}_j$, $\{m_j\}_j$, $\{n_j\}_j$.
Note that $1,2,3$ appear infinitely many times in $\mfw$. Hence, without loss of generality, we might assume that $\mfw=[3,(2,3,2,3,\dots),1,(3,1,3,1,\dots),2,(1,2,1,2,\dots),3,\dots]$, where there are only finitely many indices among the round brackets and we denote the indices outside the round brackets by $w_{s_1}=w_1=3$, $w_{s_2}=1$, $w_{s_3}=2$, $w_{s_4}=3$ and so on. Note that we can find a $3$-cyclic subsequence of $\mfw$ given by $[w_{s_1},w_{s_2},w_{s_3},w_{s_4},\dots]=[3,1,2,3,\dots]$.
 
\textbf{Step $1$}:
  If $m_0 - l_0 \geq \frac{k}{x_0}$, since the first mutation is $\delta_1$, by \Cref{lem: internal condition}, we have the mutated number $n_1 \in [l_0, m_0]$.  Otherwise, that is $m_0 - l_0 < \frac{k}{x_0}$, then we have \begin{align}n_1 < \max(l_0,m_0) + \dfrac{k}{z_1} < l_0 + \dfrac{k}{x_0} + \dfrac{k}{z_1}<l_0+\dfrac{2k}{x_0}.\end{align}
Furthermore, let us consider the finite repeating sequence $[2,3,2,3,\dots]$ after $n_{1}$. Note that $l_0$ is fixed since we pass the mutation $\delta_1$. If $n_1 - l_0 \geq  \frac{k}{x_0}$, then by \Cref{lem: internal condition}, the mutated number $m_2$ is again the internal division point of $l_0$ and $n_1$, which implies that $m_2 \leq n_1$.
  Otherwise, that is $n_1 - l_0 <  \frac{k}{x_0}$, we have \begin{align}m_2 < \max(n_1,l_0) + \dfrac{k}{y_2} < l_0 + \dfrac{k}{x_0} + \dfrac{k}{y_2} < l_0 + \dfrac{2k}{x_0}.\end{align}
  Therefore, we now claim that $\{n_i\}_{i<s_2}$ and $\{m_j\}_{j<s_2}$ are smaller than $\max(m_0,l_0 + \frac{2k}{x_0})$. In fact, we can take the induction under $s_2$. Without loss of generality, we might only prove that the claim holds for the mutated number $n_{i+1}$, that is $(l_{i+1},m_{i+1},n_{i+1})=\delta_{3}(l_i,m_i,n_i)$. If $m_i - l_0 \geq \frac{k}{x_0}$, then by \Cref{lem: internal condition}, $n_{i+1}$ is again the internal division point between $l_0$ and $m_i$. Thus, it implies that \begin{align}n_{i+1} \leq m_i\leq \max(m_0,l_0 + \frac{2k}{x_0}).\end{align}
  Otherwise, that is  $m_i - l_0 < \frac{k}{x_0}$, then we have \begin{align}n_{i+1} < \max(m_i,l_0) + \dfrac{k}{z_{i+1}} < l_0 + \dfrac{k}{x_0} + \dfrac{k}{z_{i+1}} = l_0 + \dfrac{2k}{x_0},\end{align} which also satisfies the condition.
  Thus, by induction, we conclude that $\{n_i\}_{i<s_2}$ and $\{m_j\}_{j<s_2}$ are both smaller than $\max(m_0,l_0 + \frac{2k}{x_0})$.
  
\textbf{Step $2$}:
  Now, we consider the next mutation $\delta_{w_{s_2}}=\delta_1$, that is to say $(l_{s_2},m_{s_2},n_{s_2})=\delta_1(l_{s_2-1},m_{s_2-1},n_{s_2-1})$. According to Step $1$, it is clear that \begin{align}l_{s_2}\leq \max(m_0 + \frac{k}{x_{s_2}},l_0 + \frac{2k}{x_0} + \frac{k}{x_{s_2}})\leq \max(m_0 + \frac{2k}{y_{s_2-1}},l_0 + \frac{2k}{x_0} + \frac{2k}{y_{s_2-1}}).\end{align}
  Furthermore, consider the next mutated number $n_{s_2+1}$. If $l_{s_2}-m_{s_2-1}\geq \frac{k}{y_{s_2-1}}$, then $n_{s_2+1}\in [m_{s_2-1},l_{s_2}]$. Otherwise, that is $l_{s_2}-m_{s_2-1}<\frac{k}{y_{s_2-1}}$, we have  \begin{align}\begin{array}{ll}n_{s_2+1} & \leq \max(l_{s_2},m_{s_2-1})+ \frac{k}{z_{s_2+1}}\\ & \leq m_{s_2-1}+ \frac{k}{y_{s_2-1}}+ \frac{k}{z_{s_2+1}}\\
  & \leq \max(m_0 + \frac{2k}{y_{s_2-1}},l_0 + \frac{2k}{x_0} + \frac{2k}{y_{s_2-1}}).\end{array}\end{align}
  Therefore, similar to the discussion above, by induction, we can conclude that  $\{l_i\}_{i<s_3}$ and $\{n_j\}_{j<s_3}$ are smaller than $\max(m_0 + \frac{2k}{y_{s_2-1}}, l_0 + \frac{2k}{x_0} + \frac{2k}{y_{s_2-1}})$.

\textbf{Step $3$}: Inductively, by the similar arguments as above, we can show that $\{l_j\}_j$, $\{m_j\}_j$, $\{n_j\}_j$ are all smaller than \begin{align}\max(m_0 + \dfrac{2k}{y_{s_2-1}} + \dfrac{2k}{z_{s_3-1}} + \dfrac{2k}{x_{s_4-1}} + \cdots,l_0 + \dfrac{2k}{x_0} + \dfrac{2k}{y_{s_2-1}} + \dfrac{2k}{z_{s_3-1}} + \dfrac{2k}{x_{s_4-1}} + \cdots).\end{align}
  By \Cref{lem:fibonacci-estimate}, the mutated numbers corresponding to the $3$-cyclic subsequence are bounded above by the reciprocals in the Fibonacci sequence. That is to say, 
  \begin{align}
    \left\{
      \begin{array}{ll}
        m_0 + \frac{2k}{y_{s_2-1}} + \frac{2k}{z_{s_3-1}} + \frac{2k}{x_{s_4-1}} + \cdots < m_0 + \frac{2k}{1} + \frac{2k}{1} + \frac{2k}{2}  + \cdots =m_0+2k\sum_{n=1}^{\infty}\frac{1}{F_n} \\
        l_0 + \frac{2k}{x_{s_1-1}} + \frac{2k}{y_{s_2-1}} + \frac{2k}{z_{s_3-1}} + \frac{2k}{x_{s_4-1}} + \cdots < l_0 + \frac{2k}{1} + \frac{2k}{1} + \frac{2k}{2} + \frac{2k}{3} + \cdots=l_0+ 2k\sum_{n=1}^{\infty}\frac{1}{F_n} \\
      \end{array}
    \right. .\label{eq: two series}
  \end{align}
  According to \Cref{prop: Fibonacci converge}, we obtain that the series on the right hand side of \eqref{eq: two series} converge.
  Therefore, the sequences $\{l_j\}_j$, $\{m_j\}_j$, $\{n_j\}_j$ are all bounded above, which implies that the sequence $\{\max(l_j,m_j,n_j)\}_{j=1}^{+\infty}$ is bounded above.
\end{proof}
\begin{example}\label{ex: bounded above}
	Let us consider the special $3$-cyclic mutation sequence $\mfw=[1,2,3,1,2,3,\dots]$.
  Then, by a direct calculation, the corresponding mutation chain in $\mcE$ forms a Fibonacci-type sequence.
  Recall that the mutation on the comparison triples $(l_j,m_j,n_j)$ can be described visually as moving the internal division point right. 
  Thus, we have \begin{align}l_{s} < \max(m_0,n_0) + \dfrac{k}{1} + \dfrac{k}{1} + \dfrac{k}{2} + \dfrac{k}{3} + \dfrac{k}{5} + \dfrac{k}{8} + \cdots \ (s\gg 0).\end{align} By \Cref{lem: Fibonacci property}, we may conclude that the series on the right hand side converge, which implies that $l_s$ is bounded above. 
  Note that the same argument also holds for $m_s$ and $n_s\ (s\gg 0)$.
\end{example}
\begin{corollary}\label{cor: min converge} Let the conditions be the same as above. Then, the sequence of the minimum of $(l_i,m_i,n_i)$ converges. That is to say, 
	\begin{align}
		\lim_{i\rightarrow+\infty}\min(l_i,m_i,n_i)< +\infty.
	\end{align}
\end{corollary}
\begin{proof}
	Note that the sequence $\{\min(l_i,m_i,n_i)\}_{i=0}^{+\infty}$ is a monotonically increasing sequence. Hence, by \Cref{prop:comparison-bounded} and the monotone convergence theorem, this sequence is convergent. 
\end{proof}
\begin{lemma}\label{lem: total length}
	Let $p\in \mbR_{+}$. Take any $k$-generalized Euclid tree $\mcK$ with the initial triple $(A,B,C)$ and the classical Euclid tree $\mcE$ with the initial triple $(a,b,c)$. Assume that $(x,y,z)$ is a triple in $\mcE$, such that $x,y,z>\frac{k}{p}$. Denote the corresponding triple in $\mcK$ by $(X,Y,Z)$ and the comparison triple by $(l,m,n)$. Then, the following statements hold.
	\begin{enumerate}
		\item If $|l-m|<p$ and $(l,m,n^{\prime})=\delta_3(l,m,n)$, then $\max(|l-m|,|m-n^{\prime}|,|l-n^{\prime}|)<2p$.
		\item If $|l-n|<p$ and $(l,m^{\prime},n)=\delta_2(l,m,n)$, then $\max(|l-m^{\prime}|,|m^{\prime}-n|,|l-n|)<2p$.
		\item If $|n-m|<p$ and $(l^{\prime},m,n)=\delta_1(l,m,n)$, then $\max(|l^{\prime}-m|,|m-n|,|l^{\prime}-n|)<2p$.
	\end{enumerate}
\end{lemma}
\begin{proof}
	Without loss of generality, we might only consider the case $(2)$ since others are similar. Note that $m^{\prime}=\frac{x}{x + z}\times (l + \frac{k}{x + z}) + \frac{z}{x + z}\times (n + \frac{k}{x + z})$. It means that $m^{\prime}$ is the internal division point of the interval between $l+\Delta$ and $n+\Delta$, where $\Delta=\frac{k}{x + z}$. (See \Cref{cluster real}). Hence, we have $\Delta<\frac{k}{x}<p$, which implies that $\max(|l-m^{\prime}|,|m^{\prime}-n|,|l-n|)<2p$. 
\end{proof}
\begin{lemma}\label{lem: total length bound}
	Let the conditions be the same as above. If the total length of the intervals between $l,m,n$ is less than $2p$, then it always  holds under any mutation $\delta_i$. 
\end{lemma}
\begin{proof}
	Firstly, we have $\max(|l-m|,|m-n|,|l-n|)<2p$. Without loss of generality, we might assume that $l\leq m\leq n$, which means that $m-l<2p$. By the mutation rules of the classical Euclid tree, the property that $x,y,z>\frac{k}{p}$ will always keep under the mutations.
	\begin{enumerate}
		\item Suppose that the next mutation is $\delta_1$. If $n-m\geq p>\frac{k}{y}$, then by \Cref{lem: internal condition}, $l^{\prime}$ lies in the interval between $m$ and $n$. It is clear that the total length of the new intervals becomes smaller. If $n-m< p$, then according to \Cref{lem: total length}, we have $\max(|l^{\prime}-m|,|m-n|,|l^{\prime}-n|)<2p$.
		\item Suppose that the next mutation is $\delta_2$. If $n-l\geq p>\frac{k}{x}$, then by \Cref{lem: internal condition}, $m^{\prime}$ lies in the interval between $l$ and $n$. Hence, the total length of the new intervals
 keeps invariant. If $n-l< p$, then according to \Cref{lem: total length}, we have $\max(|l-m^{\prime}|,|m^{\prime}-n|,|l-n|)<2p$.
		\item Suppose that the next mutation is $\delta_3$. If $m-l\geq p>\frac{k}{x}$, then by \Cref{lem: internal condition}, $n^{\prime}$ lies in the interval between $l$ and $m$. Hence, the total length of the new intervals becomes smaller. If $m-l< p$, then by \Cref{lem: total length}, we have $\max(|l-m|,|m-n^{\prime}|,|l-n^{\prime}|)<2p$.
	\end{enumerate}
In conclusion, the total length of the intervals under the mutations is always less that $2p$.
\end{proof}

\begin{proposition}\label{prop:comparison-convergence}
  Take any $k$-generalized Euclid tree $\mcK$ with the initial triple $(A,B,C)$ and the classical Euclid tree $\mcE$ with the initial triple $(a,b,c)$. Let $\mfw=[w_1,
  \dots,w_n,\dots]$ be an infinite reduced mutation sequence.
  Assume that $1,2,3$ all appear infinitely many times in $\mfw$.
  Then, there exists a real number $q$, such that each component of the triple in $\mcK$ converges to $q$ times of the corresponding component of the triple in $\mcE$ when $n$ goes infinity. That is to say, 
  \begin{align}
    \lim_{n\rightarrow +\infty}\frac{\mcM^{\mfw_n;k}(A,B,C)}{\mcM^{\mfw_n}(a,b,c)}=q.
      \end{align}
\end{proposition}

\begin{proof}  In fact, this is equivalent to show that all the components of the comparison triples $\{(l_j,m_j,n_j)\}_j$ converge to $q$ when $j$ goes infinity. In the following, we use the same notations as \Cref{prop:comparison-bounded}.
  
 For any positive real number $\epsilon \in \mbR_{+}$, since we have assumed that the mutations in three directions appear infinite times, each component of triples in the classical Euclid chain keeps strictly increasing.
  Therefore, there exists $j_0 \in \mbN$, such that $x_j,y_j,z_j > \frac{k}{\epsilon}$ for any $j \geq j_0$. Now, we claim that the total length of the intervals will be small enough under certain mutations, that is 
  \begin{align}
    \forall \epsilon > 0, \exists j_1(\geq j_0) \in \mbN, \text{ s.t. } \forall j \in \mbN, j \geq j_3,\  \max\{|m_j-l_j|, |n_j-m_j|, |n_j-l_j|\} < 2 \epsilon. 
  \end{align}
  In fact, by \Cref{lem: total length bound}, we can assume that the initial total length of the intervals indexed by $i_0$ is larger than $2\epsilon$. Without loss of generality, we might assume that $l_0\leq m_0 \leq n_0$. If $m_0-l_0\geq \epsilon$ and $n_0-m_0\geq \epsilon$, once we mutate at $\delta_1$ or $\delta_3$, the total length of the intervals will reduce at least $\epsilon$. Once we mutate at $\delta_2$, the total length keeps invariant. Note that the total length is bounded by \Cref{prop:comparison-bounded}. Hence, after finitely many mutations (at $i_1$), the length of at least one of the intervals will always be less than $\epsilon$, regardless of any subsequent mutations. More precisely, we might assume that $l_{i_1}\leq m_{i_1}\leq n_{i_1}$ and either $m_{i_1}-l_{i_1}<\epsilon$ or $n_{i_1}-m_{i_1}<\epsilon$, see \Cref{Two cases less}. And, for any $j\geq i_1$, at least one of the inequalities $|l_j-m_j|<\epsilon$, $|l_j-n_j|<\epsilon$, $|m_j-n_j|<\epsilon$ holds.
  \begin{figure}[htpb]
	\begin{tikzpicture}[scale=0.7] 
	\begin{scope}[shift={(0,0)}]
  \coordinate (A1) at (0,0);
  \coordinate (B1) at (1.5,0);
  \coordinate (C1) at (4,0);
  \draw (A1) -- (C1);
  \fill (A1) circle (2pt) node[below=3pt] {$l_{i_1}$};
  \fill (B1) circle (2pt) node[below=3pt] {$m_{i_1}$};
  \fill (C1) circle (2pt) node[below=3pt] {$n_{i_i}$};
  \draw [decorate,decoration={brace,amplitude=8pt}]
        (A1) -- (B1) node[midway, yshift=12pt] {$< \epsilon$};
    \draw [decorate,decoration={brace,amplitude=8pt}]
        (B1) -- (C1) node[midway, yshift=12pt] {$\geq  \epsilon$};
  \end{scope}
		\begin{scope}[shift={(6,0)}]
  \coordinate (A2) at (0,0);
  \coordinate (B2) at (2.5,0);
  \coordinate (C2) at (4,0);
  \draw (A2) -- (C2);
  \fill (A2) circle (2pt) node[below=3pt] {$l_{i_1}$};
  \fill (B2) circle (2pt) node[below=3pt] {$m_{i_1}$};
  \fill (C2) circle (2pt) node[below=3pt] {$n_{i_1}$};
  \draw [decorate,decoration={brace,amplitude=8pt}]
        (A2) -- (B2) node[midway, yshift=12pt] {$\geq  \epsilon$};
    \draw [decorate,decoration={brace,amplitude=8pt}]
        (B2) -- (C2) node[midway, yshift=12pt] {$< \epsilon$};
  \end{scope}

	\end{tikzpicture}
\caption{Two cases that the length of one interval is less than $\epsilon$}
\label{Two cases less}
\end{figure}
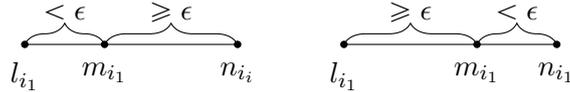
  
  Case $1$: Assume that $m_{i_1}-l_{i_1}<\epsilon$ and $n_{i_1}-m_{i_1}\geq \epsilon$. We can omit the case that we mutate at $\delta_2$ since it will not affect the total length by \Cref{lem: internal condition} (It may just transfer the Case $1$ to Case $2$ as follows). If we mutate at $\delta_3$, by \Cref{lem: total length} and \Cref{lem: total length bound}, the total length is less than $2\epsilon$ and the claim holds. If we mutate at $\delta_1$, by \Cref{lem: internal condition}, $l_{i_{1}+1}$ lies between $m_{i_{1}+1}=m_{i_1}$  and $n_{i_{1}+1}=n_{i_1}$. Next, there are two choices of the mutations $\delta_2$ and $\delta_3$.
  \begin{itemize}[leftmargin=2em]
  	\item Case $(1.1)$: Consider the next comparison triple $(l_{i_1+2},m_{i_1+2},n_{i_1+2})$ under the mutation $\delta_3$. If $l_{i_{1}+1}-m_{i_{1}+1}< \epsilon$, then by \Cref{lem: total length} and \Cref{lem: total length bound}, the total length of the intervals between $l_{i_1+2},m_{i_1+2},n_{i_1+2}$ is less than $2\epsilon$ and the claim holds. If $l_{i_{1}+1}-m_{i_{1}+1}\geq  \epsilon$, then $n_{i_1+2}$ becomes the internal point between others and we conclude that $n_{i_1+2}-m_{i_1+2}>l_{i_1+2}-n_{i_1+2}$. Indeed, we have 
  	\begin{align}
  		n_{i_1+2}=\frac{x_{i_1+1}}{x_{i_1+1} + y_{i_1+1}}\times l_{i_1+1} + \frac{y_{i_1+1}}{x_{i_1+1} + y_{i_1+1}}\times m_{i_1+1} + \frac{k}{x_{i_1+1} + y_{i_1+1}}.
  	\end{align} It implies that $n_{i_1+2}$ is obtained by dividing the interval $[m_{i_1+1},l_{i_1+1}]$ internally in the ratio $(x_{i_1+1}:y_{i_1+1})$, and  then shifting $\frac{k}{x_{i_1+1} + y_{i_1+1}}$ to the right. Note that 
  	\begin{align}
  		x_{i_1+1}=y_{i_1}+z_{i_1}>y_{i_1}=y_{i_1+1}.
  	\end{align} Hence, we have $l_{i_1+2}-n_{i_1+2}<\epsilon$. Next, if we mutate at $\delta_2$, then by \Cref{lem: total length} and \Cref{lem: total length bound}, the total length of the intervals will be less than $2\epsilon$ and the claim holds. Thus, we only need to consider the comparison triple $(l_{i_1+3},m_{i_1+3},n_{i_1+3})$ under the mutation $\delta_1$. Note that
  	\begin{align}
  		l_{i_1+3}=\frac{y_{i_1+2}}{y_{i_1+2} + z_{i_1+2}}\times m_{i_1+2} + \frac{z_{i_1+2}}{y_{i_1+2} + z_{i_1+2}}\times n_{i_1+2} + \frac{k}{y_{i_1+2} + z_{i_1+2}}.
  	\end{align} It implies that $l_{i_1+3}$ is obtained by dividing the interval $[m_{i_1+3},n_{i_1+3}]$ internally in the ratio $(z_{i_1+2}:y_{i_1+2})$, and  then shifting $\frac{k}{y_{i_1+2} + z_{i_1+2}}$ to the right. Since
  	\begin{align}
  		z_{i_1+2}=x_{i_1+1}+y_{i_1+1}>y_{i_1+1}=y_{i_1+2},
  	\end{align} we have $l_{i_1+3}-m_{i_1+3}>n_{i_1+3}-l_{i_1+3}$ and $n_{i_1+3}-l_{i_1+3}<\epsilon$. Similarly, by induction, by alternating $\delta_1$ and $\delta_2$, we always have $|n_j-l_j|<\epsilon$ for any $j\geq i_1+2$, see \Cref{Case 1 infinite}. However, according to the initial condition, $\delta_1,\delta_2,\delta_3$ will appear infinitely many times. Thus, once we mutate at $\delta_2$, by \Cref{lem: total length} and \Cref{lem: total length bound}, the total length of the intervals will be less than $2\epsilon$ and the claim holds.
  	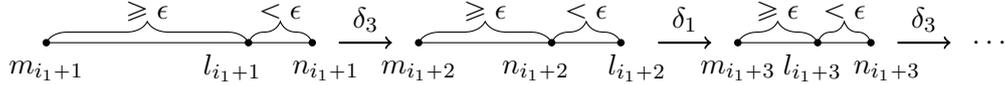
\begin{figure}[htpb]
\begin{tikzpicture}[scale=0.7] 
\begin{scope}[shift={(0,0)}]
  \coordinate (A1) at (-1,0);
  \coordinate (B1) at (2.8,0);
  \coordinate (C1) at (4,0);
  \draw (A1) -- (C1);
  \fill (A1) circle (2pt) node[below=3pt] {$m_{i_1+1}$};
  \fill (B1) circle (2pt) node[below=0.5pt, xshift=-6pt] {$l_{i_1+1}$};
  \fill (C1) circle (2pt) node[below=3pt, xshift=5pt] {$n_{i_1+1}$};
  \draw [decorate,decoration={brace,amplitude=8pt}]
        (A1) -- (B1) node[midway, yshift=12pt] {$\geq \epsilon$};
    \draw [decorate,decoration={brace,amplitude=8pt}]
        (B1) -- (C1) node[midway, yshift=12pt] {$<  \epsilon$};
  \end{scope}
  
  \begin{scope}[shift={(6,0)}]
  \coordinate (A2) at (0,0);
  \coordinate (B2) at (2.5,0);
  \coordinate (C2) at (3.8,0);
  \draw (A2) -- (C2);
  \fill (A2) circle (2pt) node[below=3pt] {$m_{i_1+2}$};
  \fill (B2) circle (2pt) node[below=3pt, xshift=-6pt] {$n_{i_1+2}$};
  \fill (C2) circle (2pt) node[below=0.5pt, xshift=6pt] {$l_{i_1+2}$};
  \draw [decorate,decoration={brace,amplitude=8pt}]
        (A2) -- (B2) node[midway, yshift=12pt] {$\geq  \epsilon$};
    \draw [decorate,decoration={brace,amplitude=8pt}]
        (B2) -- (C2) node[midway, yshift=12pt] {$< \epsilon$};
  \end{scope}

  \begin{scope}[shift={(12,0)}]
  \coordinate (A3) at (1.5,0);
  \coordinate (B3) at (0,0);
  \coordinate (C3) at (2.5,0);
  \draw (A3) -- (C3);
  \fill (A3) circle (2pt) node[below=0.5pt, xshift=-2pt] {$l_{i_1+3}$};
  \fill (B3) circle (2pt) node[below=3pt, xshift=0pt] {$m_{i_1+3}$};
  \fill (C3) circle (2pt) node[below=3pt, xshift=6pt] {$n_{i_1+3}$};
  \draw (A3) -- (B3);
    \draw [decorate,decoration={brace,amplitude=8pt}]
        (B3) -- (A3) node[midway, yshift=12pt] {$\geq  \epsilon$};
        \draw [decorate,decoration={brace,amplitude=8pt}]
        (A3) -- (C3) node[midway, yshift=12pt] {$< \epsilon$};
  \end{scope}
  \draw[->, thick] (4.5,0) -- (5.5,0) node[midway, above=0pt] {$\delta_3$};  \draw[->, thick] (10.5,0) -- (11.5,0)node[midway, above=0pt] {$\delta_1$};
  \draw[->, thick] (15,0) -- (16,0)node[midway, above=0pt] {$\delta_3$} node[right=4pt] {$\cdots$};
   \end{tikzpicture}
  \caption{Alternate mutations of $\delta_1$ and $\delta_3$}
  \label{Case 1 infinite}
  \end{figure}
  	\item Case $(1.2)$: Consider the next comparison triple $(l_{i_1+2},m_{i_1+2},n_{i_1+2})$ under the mutation $\delta_2$. If $n_{i_1+1}-l_{i_1+1}< \epsilon$, then by \Cref{lem: total length} and \Cref{lem: total length bound}, the total length of the intervals between $l_{i_1+2},m_{i_1+2},n_{i_1+2}$ will be less than $2\epsilon$ and the claim holds. If $n_{i_1+1}-l_{i_1+1}\geq \epsilon$, then according to \Cref{lem: internal condition}, $m_{i_1+2}$ becomes the internal point between $l_{i_1+2}$ and $n_{i_1+2}$. Note that 
  	\begin{align}
  		m_{i_1+2}=\frac{x_{i_1+1}}{x_{i_1+1} + z_{i_1+1}}\times l_{i_1+1} + \frac{z_{i_1+1}}{x_{i_1+1} + z_{i_1+1}}\times n_{i_1+1} + \frac{k}{x_{i_1+1} + z_{i_1+1}}.
  	\end{align} If $m_{i_1+2}-l_{i_1+2}\geq n_{i_1+2}-m_{i_1+2}$, by a direct calculation, we get 
  	\begin{align}
  		2k+(n_{i_1+1}-l_{i_1+1})(z_{i_1+1}-x_{i_1+1})\geq 0. 
  	\end{align} Since $z_{i_1+1}=y_{i_1}+z_{i_1}$ and $z_{i_1+1}=z_{i_1}$, we obtain that $n_{i_1+1}-l_{i_1+1}\leq \frac{2k}{y_{i_1}}<2\epsilon$. Then, by \Cref{lem: total length bound}, the claim holds. Hence, we only need to consider the case that $m_{i_1+2}-l_{i_1+2}< n_{i_1+2}-m_{i_1+2}$, where $m_{i_1+2}-l_{i_1+2}<\epsilon$ and $n_{i_1+2}-m_{i_1+2}\geq \epsilon$. Next, if we mutate at $\delta_3$, then by \Cref{lem: total length} and \Cref{lem: total length bound}, the total length will be less than $2\epsilon$ and the claim holds. Thus, we consider $(l_{i_1+3},m_{i_1+3},n_{i_1+3})$ under the mutation $\delta_1$. By the same discussion, if $l_{i_1+3}-m_{i_1+3}\geq n_{i_1+3}-l_{i_1+3}$, then $n_{i_1+3}-m_{i_1+3}<2\epsilon$ and the claim holds. If $l_{i_1+3}-m_{i_1+3}<n_{i_1+3}-l_{i_1+3}$, we have $l_{i_1+3}-m_{i_1+3}<\epsilon$. By induction, by alternating $\delta_1$ and $\delta_2$, we always have $|l_j-m_j|<\epsilon$ for any $j\geq i_1+2$, see \Cref{Case 2 infinite}. However, according to the initial condition, $\delta_1,\delta_2,\delta_3$ will appear infinitely many times. Thus, once we mutate at $\delta_3$, by \Cref{lem: total length} and \Cref{lem: total length bound}, the total length of the intervals will be less than $2\epsilon$ and the claim holds.
  \end{itemize} 
  \begin{figure}[htpb]
\begin{tikzpicture}[scale=0.7] 
\begin{scope}[shift={(0,0)}]
  \coordinate (A1) at (-1,0);
  \coordinate (B1) at (0.2,0);
  \coordinate (C1) at (4,0);
  \draw (A1) -- (C1);
  \fill (A1) circle (2pt) node[below=3pt] {$m_{i_1+1}$};
  \fill (B1) circle (2pt) node[below=0.5pt, xshift=6pt] {$l_{i_1+1}$};
  \fill (C1) circle (2pt) node[below=3pt] {$n_{i_1+1}$};
  \draw [decorate,decoration={brace,amplitude=8pt}]
        (A1) -- (B1) node[midway, yshift=12pt] {$< \epsilon$};
    \draw [decorate,decoration={brace,amplitude=8pt}]
        (B1) -- (C1) node[midway, yshift=12pt] {$\geq  \epsilon$};
  \end{scope}
  
  \begin{scope}[shift={(6,0)}]
  \coordinate (A2) at (0.2,0);
  \coordinate (B2) at (1.5,0);
  \coordinate (C2) at (4,0);
  \draw (A2) -- (C2);
  \fill (A2) circle (2pt) node[below=0.5pt] {$l_{i_1+2}$};
  \fill (B2) circle (2pt) node[below=3pt, xshift=5pt] {$m_{i_1+2}$};
  \fill (C2) circle (2pt) node[below=3pt] {$n_{i_1+2}$};
  \draw [decorate,decoration={brace,amplitude=8pt}]
        (A2) -- (B2) node[midway, yshift=12pt] {$<  \epsilon$};
    \draw [decorate,decoration={brace,amplitude=8pt}]
        (B2) -- (C2) node[midway, yshift=12pt] {$\geq \epsilon$};
  \end{scope}

  \begin{scope}[shift={(12,0)}]
  \coordinate (A3) at (1.5,0);
  \coordinate (B3) at (0.5,0);
  \coordinate (C3) at (3,0);
  \draw (A3) -- (C3);
  \fill (A3) circle (2pt) node[below=0.5pt, xshift=4pt] {$l_{i_1+3}$};
  \fill (B3) circle (2pt) node[below=3pt, xshift=-5pt] {$m_{i_1+3}$};
  \fill (C3) circle (2pt) node[below=3pt, xshift=5pt] {$n_{i_1+3}$};
  \draw (A3) -- (B3);
    \draw [decorate,decoration={brace,amplitude=8pt}]
        (B3) -- (A3) node[midway, yshift=12pt] {$<  \epsilon$};
        \draw [decorate,decoration={brace,amplitude=8pt}]
        (A3) -- (C3) node[midway, yshift=12pt] {$\geq \epsilon$};
  \end{scope}
  \draw[->, thick] (4.5,0) -- (5.5,0) node[midway, above=0pt] {$\delta_2$};  \draw[->, thick] (10.5,0) -- (11.5,0)node[midway, above=0pt] {$\delta_1$};
  \draw[->, thick] (15.5,0) -- (16.5,0)node[midway, above=0pt] {$\delta_2$} node[right=4pt] {$\cdots$};
   \end{tikzpicture}
  \caption{Alternate mutations of $\delta_1$ and $\delta_2$}
  \label{Case 2 infinite}
  \end{figure}
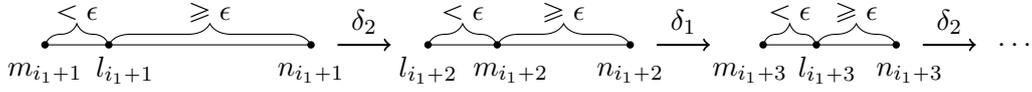
  
Case $2$: Assume that $m_{i_1}-l_{i_1}\geq \epsilon$ and $n_{i_1}-m_{i_1}< \epsilon$. If we mutate at $\delta_1$, then by \Cref{lem: total length} and \Cref{lem: total length bound}, the total length will be less than $2\epsilon$ and the claim holds. We can omit the case that we mutate at $\delta_2$ since it will not affect the total length by \Cref{lem: internal condition} (It may just transfer the Case $2$ to Case $1$ as above). The only nontrivial case is to mutate at $\delta_3$, which makes $n_{i_1+1}$ the internal point between $l_{i_1+1}=l_{i_1}$ and $m_{i_1+1}=m_{i_1}$. Luckily, the discussion is totally the same as Case $1$, which means that after finitely many mutations, the total length of the intervals will always be less than $2\epsilon$ and the claim holds.

 Finally, based on \Cref{cor: min converge} and the claim as above, we can directly obtain that each component of the comparison triple $(l_j,m_j,n_j)$ converges to a same real number $q$ as $\min(l_i,m_i,n_i)$.
\end{proof}

Here, we emphasize that in \Cref{prop:comparison-convergence}, we only consider the mutation chain where the three mutations $\delta_1,\delta_2,\delta_3$ appear infinitely  many times.
If one of the three mutations appear only finitely many times in the mutation chain, we have the following proposition.
\begin{proposition}\label{prop:comparison-convergence-two}
 Take any $k$-generalized Euclid tree $\mcK$ with the initial triple $(A,B,C)$ and the classical Euclid tree $\mcE$ with the initial triple $(a,b,c)$. Let $\mfw=[w_1,
  \dots,w_n,\dots]$ be an infinite reduced mutation sequence.
 Assume that one index $i$ of $\{1,2,3\}$ appears only finitely many times in $\mfw$.
  Then, there exists a real number $q$, such that two components of the triple indexed by $\{1,2,3\}\backslash \{i\}$ in $\mcK$ converge to $q$ times of the corresponding components of the triple in $\mcE$ when $n$ goes infinity. That is to say, set $X_j^{\mfw_n}$ to be the $j$-th component of $\mcM^{\mfw_n}_k(A,B,C)$ and $x_j^{\mfw_n}$ to be the $j$-th component of $\mcM^{\mfw_n}(a,b,c)$ with $j\neq i$, we have 
  \begin{align}
  \lim_{n\rightarrow +\infty}  \dfrac{X_j^{\mfw_n}}{x_j^{\mfw_n}}=q.
  \end{align}

\end{proposition} 

\begin{proof}
  Without loss of generality, we might assume that $i=1$ and after $w_l=1$, the subsequence is $\mfw^{(s)}=[w_{s+1},w_{s+2},w_{s+3}\dots]=[2,3,2,\dots]$. We denote the finite subsequence of $\mfw^{(s)}$ by $\mfw^{(s)}_n=[w_{s+1},w_{s+2},\dots,w_{s+n}]$. Hence, we have 
  \begin{align}
  	\mcM^{\mfw^{(s)}_n}(x_s,y_s,z_s)=\left\{ \begin{array}{cc}
 (x_s,nx_s+z_s,(n-1)x_s+z_s),\ \ \text{if $n$ is odd} \\
 (x_s,(n-1)x_s+z_s,nx_s+z_s),	\ \ \text{if $n$ is even}
 \end{array}
  	 \right.
  \end{align}  and \begin{align}
  	\mcM^{\mfw^{(s)}_n}_k(X_s,Y_s,Z_s)=\left\{ \begin{array}{cc}
 (X_s,nX_s+Z_s+nk,(n-1)X_s+Z_s+(n-1)k),\ \ \text{if $n$ is odd} \\
 (X_s,(n-1)X_s+Z_s+(n-1)k,nX_s+Z_s+nk),	\ \ \text{if $n$ is even}
 \end{array}
  	 \right. .
  \end{align}
 Thus, we can directly calculate and obtain that  
  \begin{align}
  	\lim_{n\rightarrow +\infty}\frac{nX_s+Z_s+nk}{nx_s+z_s}=\frac{X_s+k}{x_s},
  \end{align} which implies that 
  \begin{align}
  	 \lim_{n\rightarrow +\infty}  \dfrac{X_2^{\mfw_n}}{x_2^{\mfw_n}}= \lim_{n\rightarrow +\infty}  \dfrac{X_3^{\mfw_n}}{x_3^{\mfw_n}}=\frac{X_s+k}{x_s}=q.
  \end{align} Hence, the cases that $i=2,3$ are similar and the proposition holds.
\end{proof}

By combining \Cref{prop:comparison-convergence} and \Cref{prop:comparison-convergence-two}, we derive the following theorem.

\begin{theorem}\label{thm:comparison-convergence}
   Take any $k$-generalized Euclid tree $\mcK$ with the initial triple $(A,B,C)$ and the classical Euclid tree $\mcE$ with the initial triple $(a,b,c)$. Let $\mfw=[w_1,
  \dots,w_n,\dots]$ be an infinite reduced mutation sequence. Then, the statements as follows hold:
  \begin{enumerate}
  	\item If $1,2,3$ all appear infinitely many times in $\mfw$, then there exists a real number $q$, such that each component of the triple in $\mcK$ converges to $q$ times of the corresponding component of the triple in $\mcE$ when $n$ goes infinity.
  	\item If one index $i$ of $\{1,2,3\}$ appears only finitely many times in $\mfw$, then there exists a real number $q$, such that two components of the triple indexed by $\{1,2,3\}\backslash \{i\}$ in $\mcK$ converge to $q$ times of the corresponding components of the triple in $\mcE$ when $n$ goes infinity.
  \end{enumerate}
  \end{theorem}
\begin{conjecture}
	For the case $(2)$ in \Cref{thm:comparison-convergence}, it is clear that $q\in \mbQ_+$. On the other hand, for the case $(1)$ with $k\in \mbN_+$, we conjecture that $q\in \mbR_+\backslash\mbQ_+$.
\end{conjecture}
In the following, we give an example of  \Cref{thm:comparison-convergence} to illustrate how the comparison triples behave.
\begin{example}
  Consider the $7$-generalized Euclid tree $\mcK$ starting at $(1,4,9)$, and suppose that the classical Euclid tree $\mcE$ starts at $(1,1,1)$.
  Now, we take the mutation sequence $\mfw=[1,2,1,2,3,1,2,1,3,1,2,1,2]$ and we obtain a chain of triples in $\mcE$ as follows:
  \begin{align*}
  \begin{array}{ll}
    (1,1,1) \stackrel{\mcM_1}{\longrightarrow} (2,1,1) \stackrel{\mcM_2}{\longrightarrow} (2,3,1) \stackrel{\mcM_1}{\longrightarrow} (4,3,1) \stackrel{\mcM_2}{\longrightarrow} (4,5,1) \stackrel{\mcM_3}{\longrightarrow} \color{red} (4,5,9) \color{black} \stackrel{\mcM_1}{\longrightarrow} \\ (14,5,9) \stackrel{\mcM_2}{\longrightarrow} (14,23,9) 
    \stackrel{\mcM_1}{\longrightarrow} (32,23,9) \stackrel{\mcM_3}{\longrightarrow} \color{red} (32,23,55) \color{black} \stackrel{\mcM_1}{\longrightarrow} (78,23,55)  \stackrel{\mcM_2}{\longrightarrow} \\  (78,133,55) \stackrel{\mcM_1}{\longrightarrow} 
    (188,133,55) \stackrel{\mcM_2}{\longrightarrow} (188,243,55) 
    \end{array}
  \end{align*}
  The corresponding triples along $\mfw$ in $\mcK$ are given by:
  \begin{align*}
  \begin{array}{ll}
     (1,4,9) \stackrel{\mcM_{1;7}}{\longrightarrow} (20,4,9) \stackrel{\mcM_{2;7}}{\longrightarrow} (20,36,9) \stackrel{\mcM_{1;7}}{\longrightarrow} (52,36,9) \stackrel{\mcM_{2;7}}{\longrightarrow} (52,68,9) \stackrel{\mcM_{3;7}}{\longrightarrow} \color{red} (52,68,127) \color{black} \stackrel{\mcM_{1;7}}{\longrightarrow} \\  (202,68,127) \stackrel{\mcM_{2;7}}{\longrightarrow} (202,336,127) \stackrel{\mcM_{1;7}}{\longrightarrow} (470,336,127) \stackrel{\mcM_{3;7}}{\longrightarrow} \color{red} (470,336,813) \color{black} \stackrel{\mcM_{1;7}}{\longrightarrow} (1156,336,813) \\  \stackrel{\mcM_{2;7}} {\longrightarrow} (1156,1976,813) \stackrel{\mcM_{1;7}}{\longrightarrow} (2796,1976,813) \stackrel{\mcM_{2;7}}{\longrightarrow} (2796,3616,813) 
    \end{array}
  \end{align*}
  Finally, the comparison triples along $\mfw$ are given by:
  \begin{align*}
  \begin{array}{ll}
    (1,4,9) \stackrel{\delta_1}{\longrightarrow} (10,4,9) \stackrel{\delta_2}{\longrightarrow}  (10,12,9) \stackrel{\delta_1}{\longrightarrow}  (13,12,9) \stackrel{\delta_2}{\longrightarrow}  (13,13.6,9) \stackrel{\delta_3}{\longrightarrow}  \color{red} (13,13.6,14.11)  \color{black} \stackrel{\delta_1}{\longrightarrow} \\
     (14.43,13.6,14.11)\stackrel{\delta_2}{\longrightarrow}  (14.43,14.61,14.11) \stackrel{\delta_1}{\longrightarrow}  (14.69,14.61,14.11) \stackrel{\delta_3}{\longrightarrow}  \color{red} (14.69,14.61,14.78) \color{black} \stackrel{\delta_1}{\longrightarrow}\\
      (14.82,14.61,14.78) \stackrel{\delta_2}{\longrightarrow}  (14.82,14.86,14.78) \stackrel{\delta_1}{\longrightarrow}  (14.87,14.86,14.78) \stackrel{\delta_2}{\longrightarrow}  (14.87,14.88,14.78)
    \end{array}
  \end{align*}
  As we can see, the first and second components of the comparison triples get closer to each other as the mutations $\delta_1$ and $\delta$ alternately appear. On the other hand, once we process the mutation $\delta_3$, the third components of the comparison triples also become closer to the others.
\end{example}

\section{Asymptotic phenomenon of the logarithmic generalized Markov tree}\label{asym of LGMT}

In this section, we study the properties of the generalized Markov tree. Our main result is that after taking the logarithm of the generalized Markov tree, it converges to the classical Euclid tree up to a scalar multiple, which is similar as \Cref{thm:comparison-convergence}.


\subsection{Ratio number sequence}\label{sub: Ratio number sequence}

In this subsection, we aim to exhibit the asymptotic behavior of the generalized Markov triples.

Firstly, let us discuss the case when the last mutation of $\mfw$ is $\mu_1$, that is $\mfw=[w_1\dots,1]$. Assume that $a\neq  \max(a,b,c)$ and we get the next triple by $(a_1,b_1,c_1)=\mu_{1}(a,b,c) \coloneq (k_1bc,b,c)$.  Next, by processing $\mu_2$ to the triple $(k_1bc,b,c)$, we obtain that
\begin{align}
  (a_2,b_2,c_2) \coloneq \mu_2(k_1bc,b,c) = (k_1bc,\dfrac{k_1^2b^2c^2 + \lambda_2 k_1bc^2 + c^2}{b}, c).	
\end{align} 
Comparing the mutated number $b_2$ with the other two invariant numbers, that is $a_2=a_1, c_2=c_1$, we have
\begin{align}
  k_2 \coloneq \dfrac{b_2}{a_2 c_2} = \dfrac{k_1^2b^2c^2 + \lambda_2 k_1bc^2 + c^2}{k_1b^2c^2} = k_1 + \dfrac{\lambda_2}{b} + \dfrac{1}{k_1b^2} .
\end{align}
Also, if we process $\mu_3$ to the triple $(k_1bc, b, c)$, we have $(a_2,b_2,c_2) \coloneq \mu_3(k_1bc,b,c)= (a_1,b_1,k_2a_1b_1)$, where \begin{align}k_2= k_1 + \dfrac{\lambda_3}{c} + \dfrac{1}{k_1c^2}.\end{align}
Similarly, the statements for the cases that the last mutation is $\mu_2$ or $\mu_3$ also hold, that is the reduced mutation sequence is $\mfw=[w_1,\dots,2]$ or $\mfw=[w_1,\dots,3]$. 

In this way, given an arbitrary infinite reduced mutation sequence $\mfw$ with $|\mfw|=+\infty$, we can associate it with a number sequence $\{k_j\}_{j=1}^{+\infty}$. Here, we call it the \emph{ratio number sequence}. In fact, starting at the initial solution $(1,1,1)$ and by the induction, we can prove that $k_i> 1$ for any $i\in \mbN$. Also, note that by \Cref{lem: monotonicity}, the ratio sequence will become stable if the times of mutations are large enough, that is $k_i \approx k_{i+1} (i>>0)$. By a direct calculation and induction, we have the following lemma.
\begin{lemma}\label{lem: ratio increase}
	The ratio number sequence $\{k_j\}_{j=1}^{+\infty}$ is a strictly increasing sequence.
\end{lemma}

In conclusion, we have the following observation of the asymptotic transitive behavior of the generalized Markov triples.
\begin{observation}\label{p}
  Suppose that a generalized Markov triple $(a,b,c)$ is associated with the mutation sequence $\mfw$, such that $1,2,3$ all appear repeatedly in $\mfw$, then there exists a natural number $k_{\lambda}=3+\lambda_1+\lambda_2+\lambda_3$, such that the triples after $(a,b,c)$ are approximately of the form
  \begin{align}
    \left\{ 
    \begin{array}{ccc}
    \mu_1(a_i,b_i,c_i) &\approx& (k_{\lambda}b_ic_i, b_i, c_i) \\
    \mu_2(a_i,b_i,c_i) &\approx& (a_i, k_{\lambda}a_ic_i, c_i) \\
 \mu_3(a_i,b_i,c_i) &\approx& (a_i, b_i, k_{\lambda}a_ib_i)
    \end{array}
    \right.,\label{eq: app}
  \end{align}
  where the triple $(a_i,b_i,c_i)$ is the one obtained by processing an arbitrary composition of cluster mutations on $(a,b,c)$. 
\end{observation}
\begin{remark}
	 Note that for each generalized Markov triple $(a,b,c)$, only two equalities in \eqref{eq: app} will appear. In fact, it is direct by \Cref{lem: monotonicity}  because the cluster mutation is involutive and we always assume that $\mfw$ is reduced.
\end{remark}
\begin{example}
  Let us consider the classical Markov triple $(194,13,5)$ which is obtained from the composition of cluster mutations on $(1,1,1)$, that is 
  $(194,13,5)= \mu_1\circ \mu_2\circ \mu_3\circ \mu_2(1,1,1)$. Then, we draw a branch of Markov tree after the triple $(194,13,5)$ as the following. 
\begin{figure}[htpb] 
  \tikzset{every picture/.style={line width=0.75pt}} 
  \begin{tikzpicture}[x=0.75pt,y=0.75pt,yscale=-1,xscale=1]
  \path (50,90); 
  \draw    (212.4,147.6) -- (247.17,132.21) ;
  \draw [shift={(249,131.4)}, rotate = 156.12] [color={rgb, 255:red, 0; green, 0; blue, 0 }  ][line width=0.75]    (10.93,-3.29) .. controls (6.95,-1.4) and (3.31,-0.3) .. (0,0) .. controls (3.31,0.3) and (6.95,1.4) .. (10.93,3.29)   ;
  \draw    (213.2,160.8) -- (248.79,177.74) ;
  \draw [shift={(250.6,178.6)}, rotate = 205.45] [color={rgb, 255:red, 0; green, 0; blue, 0 }  ][line width=0.75]    (10.93,-3.29) .. controls (6.95,-1.4) and (3.31,-0.3) .. (0,0) .. controls (3.31,0.3) and (6.95,1.4) .. (10.93,3.29)   ;
  \draw    (380.6,127) -- (426.67,117.13) ;
  \draw [shift={(428.6,116.6)}, rotate = 164.69] [color={rgb, 255:red, 0; green, 0; blue, 0 }  ][line width=0.75]    (10.93,-3.29) .. controls (6.95,-1.4) and (3.31,-0.3) .. (0,0) .. controls (3.31,0.3) and (6.95,1.4) .. (10.93,3.29)   ;

  \draw (130,146) node [anchor=north west][inner sep=0.75pt]   [align=left] {(194,13,5)};
  \draw (270,122) node [anchor=north west][inner sep=0.75pt]   [align=left] {(194,2897,5)};
  \draw (270,169.6) node [anchor=north west][inner sep=0.75pt]   [align=left] {(194,13,7561)};
  \draw (216,120) node [anchor=north west][inner sep=0.75pt]   [align=left] {$\mu_2$};
  \draw (215,170) node [anchor=north west][inner sep=0.75pt]   [align=left] {$\mu_3$};
  \draw (435,107.2) node [anchor=north west][inner sep=0.75pt]   [align=left] {(43261,2897,5)};
  \draw (396,102) node [anchor=north west][inner sep=0.75pt]   [align=left] {$\mu_1$};
  \draw (400,175) node [anchor=north west][inner sep=0.75pt]   [align=left] {$\cdots$};
  \end{tikzpicture}
  	\caption{A local branch of the classical Markov tree}
  	\label{fig: Markov tree}
\end{figure}
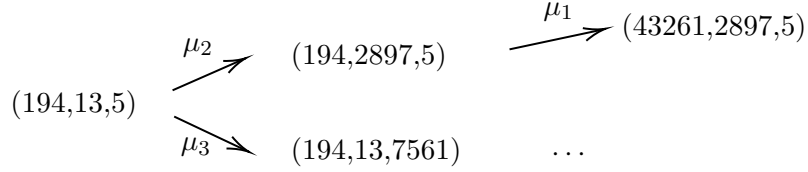 
  Let $(X,Y,Z)\coloneq (194,13,5)$, $(X_1,Y_1,Z_1) \coloneq \mu_2(X,Y,Z),(X_2,Y_2,Z_2) \coloneq \mu_1\circ \mu_2(X,Y,Z)$ and $(X_3,Y_3,Z_3) \coloneq \mu_3(X,Y,Z)$.
  By doing some simple calculation, we have \begin{align*}\frac{X}{YZ} \approx 2.98461538,\ \frac{Y_1}{X_1Z_1} \approx 2.98659794,\ \frac{X_2}{Y_2Z_2} \approx 2.98660683,\ \frac{Z_3}{X_3Y_3} \approx 2.99801745.\end{align*}
  One can observe that the phenomenon of \cref{p} has already occurred even when the length of the mutation sequence is $|\mfw|=5$, that is 
  \begin{align}
    \left\{ 
    \begin{array}{ccc}
    \mu_2(X,Y,Z) &\approx& (X,3XZ,Z) \\
    \mu_3(X,Y,Z) &\approx& (X,Y,3XY) \\
    \mu_1(X_1,Y_1,Z_1) &\approx& (3Y_1Z_1,Y_1,Z_1)
    \end{array}
    \right. .
  \end{align}
 Hence, for the Markov equation, we have $\lambda_1=\lambda_2=\lambda_3=0$ and $k_{\lambda}=3$.
\end{example}

\subsection{Generalized Euclid tree arising from logarithmic generalized Markov tree}

We have seen in the previous subsection that the cluster mutations on generalized Markov triples are approximately multiplications when the length of mutation sequence $\mfw$ is large enough. Then, it is natural to investigate the behavior of cluster mutations when we take logarithm of the generalized Markov tree, whose generalized Markov triples $(a,b,c)$ are replaced by $(\log(a),\log(b),\log(c))$.
We call such tree the \emph{logarithmic generalized Markov tree}.

In the following, we will see that the asymptotic phenomenon between the logarithmic generalized Markov tree and the $k$-generalized Euclid tree.

For brevity, we denote by $\widebar{x} \coloneq \log(x)$ for any positive real number $x$. Therefore, we have \begin{align}
 	(\widebar{a},\widebar{b},\widebar{c})\coloneq (\log(a),\log(b),\log(c)).
 \end{align}
Take any mutation chain from the generalized Markov tree. For example, if we take the mutation chain as 
\begin{align}
  (a_0, b_0, c_0) \xrightarrow{\mu_1} (a_1,b_1,c_1) \xrightarrow{\mu_2} (a_2,b_2,c_2) \xrightarrow{\mu_3} (a_3,b_3,c_3) \cdots 
\end{align} 
Then, the corresponding mutation chain in the logarithmic generalized Markov tree can be written as 
\begin{align}
  (\widebar{a_0}, \widebar{b_0}, \widebar{c_0}) \xrightarrow{\widebar{\mu_1}} (\widebar{a_1},\widebar{b_1},\widebar{c_1}) \xrightarrow{\widebar{\mu_2}} (\widebar{a_2},\widebar{b_2},\widebar{c_2}) \xrightarrow{\widebar{\mu_3}} (\widebar{a_3},\widebar{b_3},\widebar{c_3}) \cdots 
\end{align}
Here, we replace $\mu_i$ by $\widebar{\mu_i}$ to denote the mutation in the logarithmic generalized Markov tree.

More generally, given any generalized Markov triple $(a_i,b_i,c_i)$ with $a_i \neq \max(a_i,b_i,c_i)$.
Recall that according to \Cref{sub: Ratio number sequence}, if we mutate at $\mu_1$, we have \begin{align}(a_{i+1},b_{i+1},c_{i+1})\coloneq \mu_1(a_i,b_i,c_i)= (k_{i+1}\times b_ic_i,b_i,c_i),\end{align} 
which implies that \begin{align}(\widebar{a_{i+1}}, \widebar{b_{i+1}},\widebar{c_{i+1}})= (\log(a_{i+1}),\log(b_{i+1}), \log(c_{i+1}))= (\widebar{k_{i+1}}+\widebar{b_{i}}+\widebar{c_{i}}, \widebar{b_{i}}, \widebar{c_i}).\end{align} Hence, the mutation $\widebar{\mu_1}$ can be written as 
\begin{align}
  \widebar{\mu_1}(\widebar{a_i}, \widebar{b_i}, \widebar{c_i}) = (\widebar{k_{i+1}}+\widebar{b_{i}}+\widebar{c_{i}}, \widebar{b_{i}}, \widebar{c_i}).
\end{align}
Note that the other mutations $\widebar{\mu_2}$ and $\widebar{\mu_3}$ have the same phenomenon as $\widebar{\mu_1}$.

On the other hand, we may obtain the approximation phenomenon as follows.
\begin{observation}\label{ob2}
	Let the conditions be same as \Cref{p} and denote by $\widebar{k_{\lambda}}=\log(k_{\lambda})$. Then, the relation \eqref{eq: app} is equivalent to 
	\begin{align}
		\left\{ 
    \begin{array}{ccc}
    \widebar{\mu_1}(\widebar{a_i},\widebar{b_i},\widebar{c_i}) &\approx& (\widebar{k_{\lambda}} + \widebar{b_i} + \widebar{c_i}, \widebar{b_i}, \widebar{c_i}) \\
    \widebar{\mu_2}(\widebar{a_i},\widebar{b_i},\widebar{c_i}) &\approx& (\widebar{a_i}, \widebar{k_{\lambda}} + \widebar{a_i} + \widebar{c_i}, \widebar{c_i}) \\
    \widebar{\mu_3}(\widebar{a_i},\widebar{b_i},\widebar{c_i}) &\approx& (\widebar{a_i}, \widebar{b_i}, \widebar{k_{\lambda}} + \widebar{a_i} + \widebar{b_i} )
    \end{array}
    \right. .\label{eq: log}
	\end{align} 
\end{observation}

In this way, the logarithmic generalized Markov tree is asymptotically the same as the $\widebar{k_{\lambda}}$-generalized Euclid tree. To justify this statement, essentially, we need to prove that the number sequence $\{k_j\}_{j=1}^{+\infty}$ converges to $k_{\lambda}=3+\lambda_1+\lambda_2+\lambda_3$.
\begin{theorem}\label{thm: converge}
	Let $\mfw=[w_1,\dots,w_n,\dots]$ be an infinite reduced mutation sequence and $\{k_j\}_{j=1}^{+\infty}$ be the ratio number sequence associated with $\mfw$. Then, the following statements hold:
	\begin{enumerate}
		\item If $1,2,3$ all appear infinitely many times in $\mfw$, then $\displaystyle \lim_{j\to +\infty} k_j = k_{\lambda}$.
		\item If one index $i$ of $\{1,2,3\}$ appears only finitely many times in $\mfw$, then there exists a real number $k_{\beta}$, such that $\displaystyle \lim_{j\to +\infty} k_j = k_{\beta}$.
	\end{enumerate}
\end{theorem}
\begin{proof}
	Firstly, by \Cref{lem: ratio increase}, for any infinite reduced sequence $\mfw$, the corresponding ratio number sequence $\{k_j\}_{j=1}^{+\infty}$ is a strictly increasing sequence. Assume that $(x_1,x_2,x_3)$ is an arbitrary solution to the generalized Markov equation \eqref{eq: GME}. According to the mutation rules \eqref{mutation rule}, there are three possible ratio numbers. Without loss of generality, we may consider the ratio $k_j$ under the mutation $\mu_1$ as follows:
	\begin{align}
		k_j = \dfrac{x_2^2+\lambda_1x_2x_3+x_3^2}{x_1x_2x_3} = (3+\lambda_1+\lambda_2+\lambda_3)-\dfrac{x_1^2+\lambda_2x_1x_3+\lambda_3x_1x_2}{x_1x_2x_3}.\label{eq: k_j}
	\end{align} It is direct that $k_j\leq 3+\lambda_1+\lambda_2+\lambda_3$ for any $j$, which means that the sequence $\{k_j\}_{j=1}^{+\infty}$ is bounded above. Hence, by the monotone convergence theorem, it converges to some real number. Therefore, the statement $(2)$ holds. 
	
	Now, we consider the statement $(1)$. Since $1,2,3$ all appear infinitely many times in $\mfw$, by \Cref{lem: monotonicity}, $x_1,x_2,x_3$ will tend to $+\infty$ when the times of mutation are large enough. Note that for the mutation $\mu_1$, we have $x_1\neq \max(x_1,x_2,x_3)$ and 
	\begin{align}
		\dfrac{x_1^2}{x_1x_2x_3}=\dfrac{x_1}{x_2x_3}\leq \max(\dfrac{1}{x_2},\dfrac{1}{x_3}). 
	\end{align} Hence, we obtain that $k_j$ in \eqref{eq: k_j} tends to $k_{\lambda}=3+\lambda_1+\lambda_2+\lambda_3$. Similarly, we may conclude that the ratios under the mutations $\mu_2$ and $\mu_3$ also behave so. Hence, we have $\displaystyle \lim_{j\to +\infty} k_j = k_{\lambda}$.
\end{proof}
\begin{example}\label{ex: Fibonacci and Markov} We consider the Markov equation $(\lambda_1=\lambda_2=\lambda_3=0)$. Take the reduced mutation sequence $\mfw=[1,2,1,2,1,2,1]$ and we get the corresponding Markov triples as follows:
\begin{align*}
\begin{array}{cc}
	(1,1,1)\stackrel{\mu_1}{\longrightarrow} (2,1,1)\stackrel{\mu_2}{\longrightarrow} (2,5,1) \stackrel{\mu_1}{\longrightarrow} (13,5,1) \stackrel{\mu_2}{\longrightarrow}(13,34,1) \stackrel{\mu_1}{\longrightarrow} (89,34,1) \stackrel{\mu_2}{\longrightarrow} (89,233,1).
	\end{array}
\end{align*} Then, by a direct calculation, we have $k_1=2,k_2=2.5,k_3=2.6,k_4\approx2.615,k_5\approx 2.618, k_6\approx 2.618.$ In fact, let $\mfw=[w_1,w_2,w_3,w_4,\dots]=[1,2,1,2,,\dots]$, where $1$ and $2$ always alternately appear. Then, by the Catalan's identity with $r=2$ in \Cref{lem: Fibonacci property} and the cluster mutation rules, we obtain the corresponding Markov triples as
	\begin{align}
  	\mu^{\mfw_j}(1,1,1)=\left\{ \begin{array}{cc}
 (F_{2j+1},F_{2j-1},1),\ \ \text{if $j$ is odd} \\
 (F_{2j-1},F_{2j+1},1),	\ \ \text{if $j$ is even}
 \end{array}
  	 \right. .
  \end{align} Hence, by a direct calculation, we have 
  \begin{align}
  \lim_{j\rightarrow +\infty}	k_j=\lim_{j\rightarrow +\infty}\frac{F_{2j+1}}{F_{2j-1}}=\varphi^2=\frac{3+\sqrt{5}}{2}\approx 2.618.
  \end{align}
	
\end{example}
Intuitively, this example suggests the existence of the limit in case $(2)$. However, we are still unable to determine the exact value of the limit $k_{\beta}$ in general. This is because, although certain branches in the Markov tree exhibit a Fibonacci-type growth, the whole set of Markov numbers is much larger than the Fibonacci sequence. Hence, the following natural question arises.
\begin{question}
	For the case $(2)$ in \Cref{thm: converge}, how can we determine the real number $k_{\beta}$ that the sequence $\{k_j\}_{j=1}^{+\infty}$ converges to? 
\end{question}
\begin{corollary}\label{cor: limit}
	Let $\mfw=[w_1,\dots,w_n,\dots]$ be an infinite reduced mutation sequence and $\{\widebar{k_j}\}_{j=1}^{+\infty}$ be the logrithmic ratio number sequence associated with $\mfw$. Then, this number sequence converges to some real number.
\end{corollary}

\subsection{Main results}
Once we obtain \Cref{thm: converge}, motivated by \Cref{thm:comparison-convergence}, we have the following main theorem, which states the asymptotic phenomenon between the logarithmic generalized Markov tree and the classical Euclid tree.
\begin{theorem}\label{thm: generalized Markov tree}
  Let $\mfw=[w_1,\dots,w_n,\dots]$ be an infinite reduced mutation sequence and $\{k_j\}_{j=1}^{+\infty}$ be the ratio number sequence associated with $\mfw$.
  \begin{enumerate}
  	\item If $1,2,3$ all appear infinitely many times in $\mfw$, then there exists a real number $q$, such that the logarithmic generalized Markov chain along $\mfw$ converges to $q$ times of the corresponding classical Euclid chain when $n$ goes infinity.
  	\item If one index $i$ of $\{1,2,3\}$ appear only finitely many times in $\mfw$, then there exists a real number $q$, such that the components of the logarithmic generalized Markov chain along $\mfw$ indexed by $\{1,2,3\}\backslash \{i\}$ converge to $q$ times of the corresponding components in the classical Euclid chain when $n$ goes infinity.
  \end{enumerate}
\end{theorem}
\begin{proof}
 Firstly, we consider the case $(1)$. By \Cref{thm: converge},  the ratio number sequence  $\{k_j\}_{j=1}^{+\infty}$ converges to a real number $\beta$.
 Recall that by taking logarithm, the corresponding mutations can be written as 
  \begin{align}
    \left\{ 
    \begin{array}{ccc}
    \widebar{\mu_1}(\widebar{a_j},\widebar{b_j},\widebar{c_j}) &=& (\widebar{k_j} + \widebar{b_j} + \widebar{c_j}, \widebar{b_j}, \widebar{c_j}) \\
    \widebar{\mu_2}(\widebar{a_j},\widebar{b_j},\widebar{c_j}) &=& (\widebar{a_j}, \widebar{k_j} + \widebar{a_j} + \widebar{c_j}, \widebar{c_j}) \\
    \widebar{\mu_3}(\widebar{a_j},\widebar{b_j},\widebar{c_j}) &=& (\widebar{a_j}, \widebar{b_j}, \widebar{k_j} + \widebar{a_j} + \widebar{b_j} )
    \end{array}
    \right.
  \end{align}
  If we take $j_0$ to be large enough, then $\widebar{k_j} \in [k_1,k_2]$ for all $j \geq j_0$, where $k_1$ and $k_2$ can be chosen to be close enough. Therefore, we can show that the logarithmic generalized Markov tree has the same properties as the $\widebar{\beta}$-generalized Euclid tree, which are stated in \Cref{thm:comparison-convergence}. In fact, more precisely, we can modify the proof of \Cref{prop:comparison-bounded} via replacing the fixed number $k$ by the number sequence $\{\widebar{k_j}\}_{j=1}^{+\infty}$. Then, the similar arguments can be done in the proof of \Cref{prop:comparison-convergence} via replacing the inequality $x_j, y_j, z_j > k\slash \epsilon $ by the inequality $x_j, y_j, z_j > \widebar{k_i} \slash \epsilon$ for any $i\geq j_0$.

  Now, we consider the case $(2)$. Without loss of generality, we might assume that $i=1$ and after $w_s=1$, the subsequence is $\mfw^{(s)}=[w_{s+1},w_{s+2},w_{s+3}\dots]=[2,3,2,\dots]$. Denote the finite subsequence  of $\mfw^{(s)}$ by $\mfw^{(s)}_n=[w_{s+1},w_{s+2},\dots,w_{s+n}]$. Hence, we have 
  \begin{align}
  	\mcM^{\mfw^{(s)}_n}(x_s,y_s,z_s)=\left\{ \begin{array}{cc}
 (x_s,nx_s+z_s,(n-1)x_s+z_s),\ \ \text{if $n$ is odd} \\
 (x_s,(n-1)x_s+z_s,nx_s+z_s),	\ \ \text{if $n$ is even}
 \end{array}
  	 \right.
  \end{align}  and \begin{align}
  	\widebar{\mu}^{\mfw^{(s)}_n}(\widebar{a_s},\widebar{b_s},\widebar{c_s})=\left\{ \begin{array}{cc}
 (\widebar{a_s},n\widebar{a_s}+\widebar{c_s}+\sum_{j=l}^{l+n-1}\widebar{k_j},(n-1)\widebar{a_s}+\widebar{c_s}+\sum_{j=l}^{l+n-2}\widebar{k_j}),\ \ \text{if $n$ is odd} \\
 (\widebar{a_s},(n-1)\widebar{a_s}+\widebar{c_s}+\sum_{j=l}^{l+n-2}\widebar{k_j},n\widebar{a_s}+\widebar{c_s}+\sum_{j=l}^{l+n-1}\widebar{k_j}),	\ \ \text{if $n$ is even}
 \end{array}
  	 \right. .
  \end{align} Note that by \Cref{cor: limit}, there exists $q_0\in \mbR$, such that $\lim\limits_{k\rightarrow +\infty}\widebar{k_j}=q_0$. Hence, by the Ces\`aro mean theorem, we have 
  \begin{align}
  	\lim_{n\rightarrow +\infty}\frac{n\widebar{a_s}+\widebar{c_s}+\sum_{j=l}^{l+n-1}\widebar{k_j}}{nx_s+z_s}=\frac{\widebar{a_s}}{x_s}+\lim_{j\rightarrow +\infty}\frac{\sum_{j=l}^{l+n-1}\widebar{k_j}}{n}=q_0+\frac{\widebar{a_s}}{x_s}.
  \end{align} Therefore, the statement $(2)$ holds by taking $q=q_0+\frac{\widebar{a_s}}{x_s}$ and the theorem is proved. 
\end{proof} 
\begin{example}\label{ex: Fibonacci case}
	For the Markov equation \eqref{eq: Markov}, let us consider the reduced mutation sequence $\mfw=[w_1,w_2,w_3,w_4,\dots]=[1,2,1,2,,\dots]$, where $1$ and $2$ always alternately appear. It corresponds to the Case $(2)$ in \Cref{thm: generalized Markov tree}. Then, following \Cref{ex: Fibonacci and Markov}, we obtain the corresponding Markov triples as
	\begin{align}
  	\mu^{\mfw_n}(1,1,1)=\left\{ \begin{array}{cc}
 (F_{2n+1},F_{2n-1},1),\ \ \text{if $n$ is odd} \\
 (F_{2n-1},F_{2n+1},1),	\ \ \text{if $n$ is even}
 \end{array}
  	 \right. ,
  \end{align} and the classical Euclid triples (with the initial triple $(1,1,1)$) as
  \begin{align}
  	\mcM^{\mfw_n}(1,1,1)=\left\{ \begin{array}{cc}
 (n+1,n,1),\ \ \text{if $n$ is odd} \\
 (n,n+1,1),	\ \ \text{if $n$ is even}
 \end{array}
  	 \right. .
  \end{align} By a direct calculation, we have 
  \begin{align}
  	q=\lim_{n\rightarrow +\infty}\frac{\log(F_{2n-1})}{n}=\lim_{n\rightarrow +\infty}\frac{\log(F_{2n+1})}{n+1}=2\log(\varphi)=\log(\frac{3+\sqrt{5}}{2}).
  \end{align} Here, note that $q$ is not a rational number.
\end{example}
In the general case, such explicit expressions and formulas may not be available. Moreover, it is quite difficult to determine the precise value of $q$. We do not even know whether $q$ is a rational number or not. Nevertheless, motivated by the above example, we are led to the following rationality conjecture.
\begin{conjecture}\label{conj: non rational}
	In \Cref{thm: generalized Markov tree}, we conjecture that all such limits $q\in \mbR_+\backslash\mbQ_+$. 
\end{conjecture}
\section{More examples: Lampe's Diophantine equation}\label{example}
In this section, we provide more examples to exhibit and verify the asymptotic phenomenon discussed as above.

Based on \cite{Lam16}, there is a good relation between the (generalized) cluster algebra and the generalized Markov equation with $\lambda_1=0,\lambda_2=\lambda_3=2$, which is called \emph{Lampe equation}. 
\begin{align}
  X_1^2 + X_2^2 + X_3^2 + 2X_1 X_2 + 2X_1 X_3 = 7 X_1 X_2 X_3.\label{eq: Lam}
\end{align} If the coefficient $7$ in \eqref{eq: Lam} is replaced by any positive integer number $t$, it was proved in \cite[Theorem 5.7]{CL25} that the equation has positive integer solutions if and only if $t=7$.
Note that the corresponding (generalized) cluster mutation maps $\mu_i: \mbQ^3_{+} \to \mbQ^3_{+}$ are given by 
\begin{align}
\begin{array}{cc}
    \mu_1(x_1, x_2, x_3) &= (\dfrac{x_2^2 + x_3^2}{x_1}, x_2, x_3)\\
    \mu_2(x_1, x_2, x_3) &= (x_1, \dfrac{(x_1 + x_3)^2}{x_2}, x_3)\\
    \mu_3(x_1, x_2, x_3) &= (x_1, x_2, \dfrac{(x_1 + x_2)^2}{x_3})
 \end{array}.
\end{align}
According to \cite[Theorem 2.6 \& Lemma 2.7]{Lam16}, all the solutions to the Lampe equation, which are called \emph{Lampe triples}, can be generated by the initial solution $(1,1,1)$ through finitely many cluster mutations. They also have a tree structure, which is called \emph{Lampe tree}. A part of the Lampe tree is depicted below, see \Cref{Lampe tree}.

\begin{figure}[htpb]
\begin{flushleft}
\tikzset{every picture/.style={line width=0.75pt}} 
\begin{tikzpicture}[x=0.75pt,y=0.75pt,yscale=-1,xscale=1]
\path (50,20); 
\draw    (330,132.13) -- (330,104.75) ;
\draw [shift={(330,102.75)}, rotate = 90] [color={rgb, 255:red, 0; green, 0; blue, 0 }  ][line width=0.75]    (10.93,-3.29) .. controls (6.95,-1.4) and (3.31,-0.3) .. (0,0) .. controls (3.31,0.3) and (6.95,1.4) .. (10.93,3.29)   ;
\draw    (309.75,150.5) -- (292.38,168.68) ;
\draw [shift={(291,170.13)}, rotate = 313.69] [color={rgb, 255:red, 0; green, 0; blue, 0 }  ][line width=0.75]    (10.93,-3.29) .. controls (6.95,-1.4) and (3.31,-0.3) .. (0,0) .. controls (3.31,0.3) and (6.95,1.4) .. (10.93,3.29)   ;
\draw    (353.75,140.5) -- (377.36,148.72) ;
\draw [shift={(379.25,149.38)}, rotate = 199.19] [color={rgb, 255:red, 0; green, 0; blue, 0 }  ][line width=0.75]    (10.93,-3.29) .. controls (6.95,-1.4) and (3.31,-0.3) .. (0,0) .. controls (3.31,0.3) and (6.95,1.4) .. (10.93,3.29)   ;
\draw    (259,181.25) -- (239.75,181.14) ;
\draw [shift={(237.75,181.13)}, rotate = 0.34] [color={rgb, 255:red, 0; green, 0; blue, 0 }  ][line width=0.75]    (10.93,-3.29) .. controls (6.95,-1.4) and (3.31,-0.3) .. (0,0) .. controls (3.31,0.3) and (6.95,1.4) .. (10.93,3.29)   ;
\draw    (280.04,190.79) -- (280.13,210.43) ;
\draw [shift={(280.14,212.43)}, rotate = 269.72] [color={rgb, 255:red, 0; green, 0; blue, 0 }  ][line width=0.75]    (10.93,-3.29) .. controls (6.95,-1.4) and (3.31,-0.3) .. (0,0) .. controls (3.31,0.3) and (6.95,1.4) .. (10.93,3.29)   ;
\draw    (305.18,84.79) -- (283.99,67.96) ;
\draw [shift={(282.43,66.71)}, rotate = 38.46] [color={rgb, 255:red, 0; green, 0; blue, 0 }  ][line width=0.75]    (10.93,-3.29) .. controls (6.95,-1.4) and (3.31,-0.3) .. (0,0) .. controls (3.31,0.3) and (6.95,1.4) .. (10.93,3.29)   ;
\draw    (346.71,83) -- (361.65,66.76) ;
\draw [shift={(363,65.29)}, rotate = 132.59] [color={rgb, 255:red, 0; green, 0; blue, 0 }  ][line width=0.75]    (10.93,-3.29) .. controls (6.95,-1.4) and (3.31,-0.3) .. (0,0) .. controls (3.31,0.3) and (6.95,1.4) .. (10.93,3.29)   ;
\draw    (410.43,161.29) -- (420.4,181.49) ;
\draw [shift={(421.29,183.29)}, rotate = 243.73] [color={rgb, 255:red, 0; green, 0; blue, 0 }  ][line width=0.75]    (10.93,-3.29) .. controls (6.95,-1.4) and (3.31,-0.3) .. (0,0) .. controls (3.31,0.3) and (6.95,1.4) .. (10.93,3.29)   ;
\draw    (425.75,145.07) -- (443.84,127.81) ;
\draw [shift={(445.29,126.43)}, rotate = 136.34] [color={rgb, 255:red, 0; green, 0; blue, 0 }  ][line width=0.75]    (10.93,-3.29) .. controls (6.95,-1.4) and (3.31,-0.3) .. (0,0) .. controls (3.31,0.3) and (6.95,1.4) .. (10.93,3.29)   ;
\draw (307,133) node [anchor=north west][inner sep=0.75pt]   [align=left] {\color{red}(1,1,1)};
\draw (333,114) node [anchor=north west][inner sep=0.75pt]   [align=left] {\tiny $\mu_1$};
\draw (304,84) node [anchor=north west][inner sep=0.75pt]   [align=left] {\color{purple}(2,1,1)};
\draw (258.75,170.5) node [anchor=north west][inner sep=0.75pt]   [align=left] {\color{purple}(1,4,1)};
\draw (378.75,141) node [anchor=north west][inner sep=0.75pt]   [align=left] {\color{purple}(1,1,4)};
\draw (303,155) node [anchor=north west][inner sep=0.75pt]   [align=left] {\tiny $\mu_2$};
\draw (358,130) node [anchor=north west][inner sep=0.75pt]   [align=left] {\tiny $\mu_3$};
\draw (181.14,172.29) node [anchor=north west][inner sep=0.75pt]   [align=left] {\color{brown}(17,4,1)};
\draw (244,165) node [anchor=north west][inner sep=0.75pt]   [align=left] {\tiny $\mu_1$};
\draw (256,215) node [anchor=north west][inner sep=0.75pt]   [align=left] {\color{brown}(1,4,25)};
\draw (284,192) node [anchor=north west][inner sep=0.75pt]   [align=left] {\tiny $\mu_3$};
\draw (251.71,46.71) node [anchor=north west][inner sep=0.75pt]   [align=left] {\color{brown}(2,1,9)};
\draw (284,78) node [anchor=north west][inner sep=0.75pt]   [align=left] {\tiny $\mu_3$};
\draw (350,45.86) node [anchor=north west][inner sep=0.75pt]   [align=left] {\color{brown}(2,9,1)};
\draw (350,77) node [anchor=north west][inner sep=0.75pt]   [align=left] {\tiny $\mu_2$};
\draw (398,165) node [anchor=north west][inner sep=0.75pt]   [align=left] {\tiny $\mu_1$};
\draw (400.86,184) node [anchor=north west][inner sep=0.75pt]   [align=left] {\color{brown}(17,1,4)};
\draw (420,125) node [anchor=north west][inner sep=0.75pt]   [align=left] {\tiny $\mu_2$};
\draw (425.43,107.43) node [anchor=north west][inner sep=0.75pt]   [align=left] {\color{brown}(1,25,4)};
\end{tikzpicture}
\end{flushleft}
\caption{Lampe tree}
\label{Lampe tree}
\end{figure}
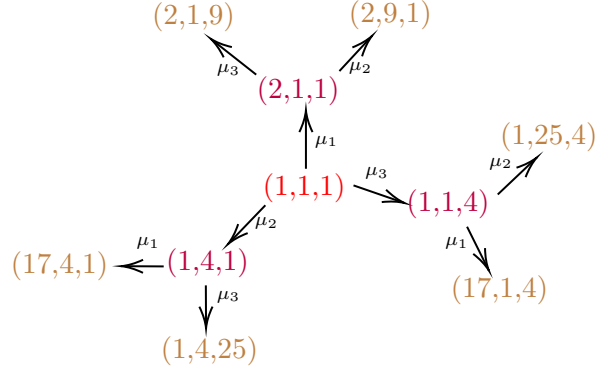
\vspace{0pt}

\begin{example}\label{ex: lampe1}
  Let us consider the Lampe triple $(29186,169,25)$ which is obtained from the composition of cluster mutations on $(1,1,1)$, that is
  $(29186,169,25)= \mu_1\circ \mu_2\circ \mu_3\circ \mu_2(1,1,1)$. Then we draw a branch of Lampe tree after the triple $(29186,169,25)$ as the following. Let $(X,Y,Z)\coloneq (29186,169,25)$, $(X_1,Y_1,Z_1) \coloneq \mu_2(X,Y,Z),\ (X_2,Y_2,Z_2) \coloneq \mu_1\circ \mu_2(X,Y,Z)$, 
  and $(X_3,Y_3,Z_3) \coloneq \mu_3(X,Y,Z)$.
  By direct calculation, we have \begin{align*}\frac{X}{YZ} \approx 6.907928994, \dfrac{Y_1}{X_1Z_1} \approx 6.9197683821,\dfrac{X_2}{Y_2Z_2} \approx 6.91976838227,\dfrac{Z_3}{X_3Y_3} \approx 6.98816061.\end{align*}
  One can observe that the phenomenon of \Cref{p} has already occurred even when the length of the mutation sequence is  $|\mfw|=5$, that is 
  \begin{align}
    \left\{ 
    \begin{array}{ccc}
    \mu_2(X,Y,Z) &\approx& (X,7XZ,Z) \\
    \mu_3(X,Y,Z) &\approx& (X,Y,7XY) \\
    \mu_1(X_1,Y_1,Z_1) &\approx& (7Y_1Z_1,Y_1,Z_1)
    \end{array}
    \right. .
  \end{align} In the following, we give an example to verify \Cref{thm: generalized Markov tree}.
\end{example}
\begin{figure}[htpb] 
  \tikzset{every picture/.style={line width=0.75pt}} 
  \begin{tikzpicture}[x=0.75pt,y=0.75pt,yscale=-1,xscale=1]
  \path (50,90); 
  \draw    (212.4,147.6) -- (247.17,132.21) ;
  \draw [shift={(249,131.4)}, rotate = 156.12] [color={rgb, 255:red, 0; green, 0; blue, 0 }  ][line width=0.75]    (10.93,-3.29) .. controls (6.95,-1.4) and (3.31,-0.3) .. (0,0) .. controls (3.31,0.3) and (6.95,1.4) .. (10.93,3.29)   ;
  \draw    (213.2,160.8) -- (248.79,177.74) ;
  \draw [shift={(250.6,178.6)}, rotate = 205.45] [color={rgb, 255:red, 0; green, 0; blue, 0 }  ][line width=0.75]    (10.93,-3.29) .. controls (6.95,-1.4) and (3.31,-0.3) .. (0,0) .. controls (3.31,0.3) and (6.95,1.4) .. (10.93,3.29)   ;
  \draw    (390.6,127) -- (426.67,117.13) ;
  \draw [shift={(428.6,116.6)}, rotate = 164.69] [color={rgb, 255:red, 0; green, 0; blue, 0 }  ][line width=0.75]    (10.93,-3.29) .. controls (6.95,-1.4) and (3.31,-0.3) .. (0,0) .. controls (3.31,0.3) and (6.95,1.4) .. (10.93,3.29)   ;

  \draw (115,146) node [anchor=north west][inner sep=0.75pt]   [align=left] {(29186,169,25)};
  \draw (257,122) node [anchor=north west][inner sep=0.75pt]   [align=left] {(29186,5049009,25)};
  \draw (255,169.6) node [anchor=north west][inner sep=0.75pt]   [align=left] {(29186,169,34468641)};
  \draw (216,120) node [anchor=north west][inner sep=0.75pt]   [align=left] {$\mu_2$};
  \draw (215,170) node [anchor=north west][inner sep=0.75pt]   [align=left] {$\mu_3$};
  \draw (435,107.2) node [anchor=north west][inner sep=0.75pt]   [align=left] {(873449321,5049009,25)};
  \draw (396,102) node [anchor=north west][inner sep=0.75pt]   [align=left] {$\mu_1$};
  \draw (450,175) node [anchor=north west][inner sep=0.75pt]   [align=left] {$\cdots$};
  \end{tikzpicture}
\caption{A local branch of the Lampe tree}
\label{fig: Lampe tree}
\end{figure}
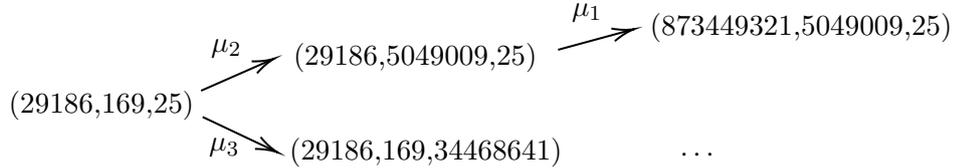 
\begin{example}\label{ex: lampe2}
  Let us consider a mutation chain $\mfw=[1,2,1,2,3,1,3,2,1,3]$ in the classical Euclid tree $\mcE$ starts at $(1,1,1)$, we obtain a chain of triples in $\mcE$ as follows:
  \begin{align*}
  \begin{array}{ll}
     (1,1,1) \stackrel{\mcM_1}{\longrightarrow} (2,1,1) \stackrel{\mcM_2}{\longrightarrow} (2,3,1) \stackrel{\mcM_1}{\longrightarrow} (4,3,1) \stackrel{\mcM_2}{\longrightarrow} (4,5,1) \stackrel{\mcM_3}{\longrightarrow} (4,5,9) \stackrel{\mcM_1}{\longrightarrow} (14,5,9) \stackrel{\mcM_3}{\longrightarrow}\\  (14,5,19)  \stackrel{\mcM_2}{\longrightarrow} (14,33,19)
     \stackrel{\mcM_1}{\longrightarrow} (52,33,19)\stackrel{\mcM_3}{\longrightarrow} (52,33,85)
     \end{array}
  \end{align*}
  The corresponding chain of Lampe triples along $\mfw$ is given as 
  \begin{align*}
  \begin{array}{ll}
     (1,1,1) \stackrel{\mu_1}{\longrightarrow} (2,1,1) \stackrel{\mu_2}{\longrightarrow} (2,9,1) \stackrel{\mu_1}{\longrightarrow} (41,9,1) \stackrel{\mu_2}{\longrightarrow} (41,196,1) \stackrel{\mu_3}{\longrightarrow} (41,196,56 169) \stackrel{\mu_1}{\longrightarrow}\\  (76 951 097,196,56 169) 
     \stackrel{\mu_3}{\longrightarrow} (76 951 097,196,105 422 946 721) \stackrel{\mu_2}{\longrightarrow}\\  (76 951 097, 56 786 879 793 920 618 169 ,105 422 946 721) 
     \stackrel{\mu_1}{\longrightarrow}\\ (41 906 481 420 650 699 762 738 336 936 066, 56 786 879 793 920 618 169, 105 422 946 721)\stackrel{\mu_3}{\longrightarrow}\\ 
      (41 906 481 420 650 699 762 738 336 936 066,56 786 879 793 920 618 169,\Omega),
     \end{array}
  \end{align*} where $\Omega=16658168261144613164154859719895467993908086960063225$.
  Now, if we take the logarithm of the Lampe triples, we have the following chain:
  \begin{align*}
  \begin{array}{ll}
     (0,0,0) \stackrel{\widebar{\mu_1}}{\longrightarrow} (0.69,0,0) \stackrel{\widebar{\mu_2}}{\longrightarrow} (0.69, 2.20, 0) \stackrel{\widebar{\mu_1}}{\longrightarrow} (3.71, 2.20, 0) \stackrel{\widebar{\mu_2}}{\longrightarrow} (3.71, 5.28, 0) \stackrel{\widebar{\mu_3}}{\longrightarrow}\\  (3.71, 5.28, 10.94)  \stackrel{\widebar{\mu_1}}{\longrightarrow} (18.16, 5.28, 10.94) \stackrel{\widebar{\mu_3}}{\longrightarrow} (18.16, 5.28, 25.38) \stackrel{\widebar{\mu_2}}{\longrightarrow} (18.16, 45.49, 25.38) \stackrel{\widebar{\mu_1}}{\longrightarrow}\\  (72.81, 45.49, 25.38)\stackrel{\widebar{\mu_3}}{\longrightarrow} (72.81, 45.49, 120.25)
    \end{array}
  \end{align*}
  Finally, we write down the comparison chain along $\mfw$ to illustrate \Cref{thm: generalized Markov tree}: 
  \begin{align*}
  \begin{array}{ll}
     (0,0,0) \stackrel{\psi_1}{\longrightarrow} (0.345,0,0) \stackrel{\psi_2}{\longrightarrow} (0.345, 0.733 , 0) \stackrel{\psi_1}{\longrightarrow} (0.928, 0.733, 0) \stackrel{\psi_2}{\longrightarrow} (0.928, 1.056, 0) \stackrel{\psi_3}{\longrightarrow} \\  (0.928, 1.056, 1.22) 
     \stackrel{\psi_1}{\longrightarrow} (1.297, 1.056, 1.22) \stackrel{\psi_3}{\longrightarrow} (1.297, 1.056, 1.336) \stackrel{\psi_2}{\longrightarrow} (1.297, 1.378, 1.336) \stackrel{\psi_1}{\longrightarrow}\\  (1.400, 1.378, 1.336) \stackrel{\psi_3}{\longrightarrow} (1.400, 1.378,1.415)
     \end{array}
  \end{align*}
  Thus, we can see that given a mutation chain, the corresponding chain of logarithmic Lampe triples converges to $q$ times of the classical Euclid triples.
   
\end{example}

\section{Generalized Markov Uniqueness Conjecture}\label{sec: GMUC}
In this section, we extend the \emph{Markov uniqueness conjecture} to the \emph{generalized Markov uniqueness conjecture}. We aim to explain how far we are  from proving them. Furthermore, we give an application of our main result (\Cref{thm: generalized Markov tree}) to this conjecture.
\subsection{Classical and generalized uniqueness conjectures}
To begin with, let us recall the famous Markov uniqueness conjecture proposed by Frobenius in 1913 as follows. For more details, we can also refer to \cite{Aig13}.
\begin{conjecture}[\emph{Uniqueness Conjecture} {\cite{Fro13}}]\label{conj: Markov uniqueness}
	If $(a,b,c)$ and $(a,b^{\prime},c^{\prime})$ are two positive integer solutions to the Markov equation \eqref{eq: Markov} with $a\geq b\geq c$ and $a\geq b^{\prime}\geq c^{\prime}$, then $b=b^{\prime}$ and $c=c^{\prime}$.
\end{conjecture}
In fact, there are several equivalent expressions of the uniqueness conjecture on the Markov equation. For example, every Markov number appears in exactly one Markov triple, up to permutation. However, note that the symmetry property in the Markov equation may not hold in the generalized Markov equation. That is to say, if $(a,b,c)$ is a Markov triple, then $(a,c,b),(b,a,c),(b,c,a),(c,a,b),(c,b,a)$ are also Markov triples. Hence, we propose a more general conjecture for the generalized Markov equations.
\begin{conjecture}[\emph{Generalized Uniqueness Conjecture}] \label{conj: CJ} If $(a,b,c)$ and $(a,b^{\prime},c^{\prime})$ are two positive integer solutions to the generalized Markov equation \eqref{eq: GME} with $a\geq b\geq c$ and $a\geq b^{\prime}\geq c^{\prime}$, then $b=b^{\prime}$ and $c=c^{\prime}$.
\end{conjecture}	
\begin{remark}
Since \Cref{conj: CJ} is proposed for arbitrary $\lambda_1,\lambda_2,\lambda_3\in \mbN$, the following five statements together with \Cref{conj: CJ} are equivalent.
	\begin{enumerate}
				\item If $(a,b,c)$ and $(a,b^{\prime},c^{\prime})$ are two positive integer solutions to the generalized Markov equation with $a\geq c\geq b$ and $a\geq c^{\prime}\geq b^{\prime}$, then $b=b^{\prime}$ and $c=c^{\prime}$.
		\item If $(a,b,c)$ and $(a^{\prime},b,c^{\prime})$ are two positive integer solutions to the generalized Markov equation with $b\geq a\geq c$ and $b\geq a^{\prime}\geq c^{\prime}$, then $a=a^{\prime}$ and $c=c^{\prime}$.
		\item If $(a,b,c)$ and $(a^{\prime},b,c^{\prime})$ are two positive integer solutions to the generalized Markov equation with $b\geq c\geq a$ and $b\geq c^{\prime}\geq a^{\prime}$, then $a=a^{\prime}$ and $c=c^{\prime}$.
		\item If $(a,b,c)$ and $(a^{\prime},b^{\prime},c)$ are two positive integer solutions to the generalized Markov equation with $c\geq a\geq b$ and $c\geq a^{\prime}\geq b^{\prime}$, then $a=a^{\prime}$ and $b=b^{\prime}$.
		\item If $(a,b,c)$ and $(a^{\prime},b^{\prime},c)$ are two positive integer solutions to the generalized Markov equation with $c\geq b\geq a$ and $c\geq b^{\prime}\geq a^{\prime}$, then $a=a^{\prime}$ and $b=b^{\prime}$.
	\end{enumerate}
\end{remark}
Note that by \Cref{tree: generalized Markov triples} or other possible examples, we can trust that this conjecture holds.

According to \Cref{thm: generalized Markov tree}, when the length of the mutation sequence $\mfw$ is large enough, the logarithmic generalized Markov tree behaves like the classical Euclid tree. Hence, it is natural to consider whether we can use this result as a method to deal with the generalized Markov conjecture since the structure of classical Euclid tree is simple and clear. Here, we call it the \emph{asymptotic method}. There are some advantages by using this method. 
\begin{enumerate}
	\item The growth rate of the generalized Markov triples under the cluster mutations is quite fast. By taking the logarithm, we can significantly reduce the growth rate and simplify it.
	\item Note that the generalized Markov conjecture is equivalent to the logarithmic generalized Markov conjecture, that is taking the logarithm on each component in \Cref{conj: CJ}.
	\item The generating rules of classical Euclid tree $\mcE$ have a beautiful lattice structure and a good connection with the Fibonacci sequence.
\end{enumerate}
   However, we think that we still have a distance away from the truth. There may be several reasons as follows:
\begin{enumerate}
	\item In \Cref{thm: generalized Markov tree}, we are not clear about the concrete value $q$ of the limit.
	\item There are still some approximation errors caused by the asymptotics and limits. 
\end{enumerate} 

\subsection{Application: an approximate method for verification}\label{subsec: Application} In this subsection, by the asymptotic method, we give an approximate way to locally verify the generalized uniqueness conjecture. With the help of the classical Euclid tree, we can roughly find where the counter-examples will appear if they exist.

\textbf{Step $1$}: Let $\mfw_n=[w_1,w_2,\dots,w_n]$ be a reduced mutation sequence with length $n$. It is direct that there are $3\times 2^{n-1}$ possible choices for such $\mfw_n$. Hence, according to \Cref{thm: generalized Markov tree}, there are also $3\times 2^{n-1}$ possible values of the limit $q$. Note that when $n$ is large enough, the numbers of $q$ may approximately fix. In fact, by experiment, when $n=6$, we already approximately get large enough possible limits $q$. (For $n\geq 7$, the added ones are still approximate to those for $n=6$.) The larger  $n$ is, the more accurate the approximation becomes. However, as a result, the calculation will become more complicated.  

\textbf{Step $2$}: For a generalized Markov triple $(a,b,c)$, without loss of generality, we may assume that $a\geq b\geq c$ and its logarithmic form is $(\log(a),\log(b),\log(c))$. If the counter-example of \Cref{conj: CJ}  exists, that is $(a,b^{\prime},c^{\prime})$ with $b^{\prime}\geq c^{\prime}$ and $(b^{\prime},c^{\prime})\neq (b,c)$, then by \Cref{thm: generalized Markov tree}, we have 
\begin{align}
	\log(a) \approx q\times x_1,
\end{align} where $x_1$ is the first component of some classical Euclid triple in $\mcE$. Then, we can fix some $n$ and there are also $3\times 2^{n-1}$ possible values of $\frac{\log(a)}{q}$. 

\textbf{Step $3$}: In the classical Euclid tree $\mcE$, find the possible triples $(x,y,z)$ whose first component $x$ is approximate to $\frac{\log(a)}{q}$ and $x\geq y\geq z$. They one-to-one correspond to some unique mutation sequence $\mfw_0$. Finally, we can check that whether the generalized Markov triple $\mu^{\mfw_0}(1,1,1)$ is a counter-example. 
\begin{remark}
	The reason why we adopt the classical Euclid tree $\mcE$ is that its structure is  simpler and more transparent. Moreover, this approach significantly reduces the computational effort required for the generalized Markov equations. However, as we aim for more precise results (as $n$ increases), the complexity will correspondingly grow. For future investigations, it is also meaningful to provide an estimate of this complexity. In conclusion, we propose an approximate method to search for the possible counter-examples of \Cref{conj: CJ}.
\end{remark}
\subsection*{Acknowledgements}
The authors  sincerely thank Ryota Akagi, Peigen Cao and Zhe Sun for their valuable discussions and insightful suggestions. They are also grateful to Yasuaki Gyoda for suggesting \Cref{ex: Fibonacci case}. In addition, Z. Chen wants to thank Xiaowu Chen, Tomoki Nakanishi, Yu Ye and Qiqi Zhao for their help and support. This work is supported by  China Scholarship Council (Grant No. 202406340022) and National Natural Science Foundation of China (Grant No. 124B2003).

\end{document}